\documentclass[12pt]{amsart}
\pdfoutput=1 

\usepackage[utf8]{inputenc}
\usepackage{amsfonts, amsmath, amssymb, mathrsfs}
\usepackage{amsthm, url}
\usepackage{kbordermatrix}
\usepackage[all]{xy}
\usepackage{graphicx, epstopdf}
\usepackage{subfig}
\usepackage{color}
\usepackage{bbm}
\usepackage{bm}
\usepackage{epigraph}
\usepackage{endnotes}
\usepackage{ifpdf}
\usepackage{enumitem}

\usepackage{multicol}
\usepackage{multirow}
\usepackage{graphicx}

\usepackage{pgf,tikz,pgfplots}
\pgfplotsset{compat=1.14}
\usepackage{mathrsfs}
\usetikzlibrary{arrows}

\usepackage[top=1in, bottom=1in, left=1in,right=1in]{geometry}
\usepackage{parskip}

\newtheorem{theorem}{Theorem}[section]

  \newtheorem{lemma}[theorem]{Lemma}

   \newtheorem{corollary}[theorem]{Corollary}
  \newtheorem{proposition}[theorem]{Proposition}

 \theoremstyle{definition}
 \newtheorem{definition}[theorem]{Definition} 
   \newtheorem{question}[theorem]{Question}
  \newtheorem{remark}[theorem]{Remark}
\newtheorem{exampleth}[theorem]{Example}
\newenvironment{example}{\begin{exampleth}}{\hfill $\diamond$\\ \end{exampleth}}
\usepackage[colorlinks=true,urlcolor=blue,linkcolor=blue,citecolor=blue]{hyperref}

\usepackage[T1]{fontenc}

\DeclareMathOperator{\GL}{GL}

\DeclareMathOperator{\Gr}{Gr}

\DeclareMathOperator{\codim}{codim}

\DeclareMathOperator{\conv}{conv}

\DeclareMathOperator{\trop}{trop}

\DeclareMathOperator{\im}{Im}

\DeclareMathOperator{\var}{var}
\DeclareMathOperator{\varbar}{\overline{var}}
\DeclareMathOperator{\init}{in}
\DeclareMathOperator{\tinit}{t-in}
\DeclareMathOperator{\spann}{span}
\DeclareMathOperator{\sgn}{sgn}
\DeclareMathOperator{\supp}{supp}
\DeclareMathOperator{\convex}{convex}

\newcommand{\R}{\mathbb{R}}
\newcommand{\Q}{\mathbb{Q}}

\newcommand{\Z}{\mathbb{Z}}

\newcommand{\C}{\mathbb{C}}

\newcommand{\m}{\mathfrak{m}}

\newcommand\ol{\overline}
\renewcommand{\P}{\mathbb{P}}

\begin{document}
\title[Positively Hyperbolic Varieties, Tropicalization, and Positroids]{Positively Hyperbolic Varieties, Tropicalization,\\ and Positroids}

\author{Felipe Rinc\'on}
\address{School of Mathematical Sciences, Queen Mary University of London, London E1 4NS, UK.}
\email{f.rincon@qmul.ac.uk}

\author{Cynthia Vinzant}
\address{North Carolina State University, Department of Mathematics,Raleigh, NC 27695, USA.}
\email{clvinzan@ncsu.edu}

\author{Josephine Yu}
\address{School of Mathematics, Georgia Tech, Atlanta GA 30332, USA.}
\email{jyu@math.gatech.edu}

%

\begin{abstract}
A variety of codimension $c$ in complex affine space is called positively hyperbolic if the imaginary part of any point in it 
does not lie in any positive linear subspace of dimension $c$.  
Positively hyperbolic hypersurfaces are defined by stable polynomials.  
We give a new characterization of positively hyperbolic varieties
using sign variations, and show that they are equivalently defined by being hyperbolic with respect
to the positive part of the Grassmannian, in the sense of Shamovich and Vinnikov.  
We prove that positively hyperbolic projective varieties have tropicalizations that are locally subfans of the type $A$ hyperplane
arrangement defined by $x_i = x_j$, 
in which the maximal cones satisfy a non-crossing condition.   
This gives new proofs of some results of  Choe--Oxley--Sokal--Wagner and Br\"and\'en on Newton polytopes and tropicalizations of stable polynomials.
We settle the question of which tropical varieties can be obtained as tropicalizations of positively hyperbolic varieties
in the case of tropical toric varieties, constant-coefficient tropical curves, and 
Bergman fans. 
Along the way, we also give a new characterization of positroids in terms of a non-crossing condition on their Bergman fans.
\end{abstract}

\maketitle 

\section{Introduction}

There are many beautiful appearances of real rooted-ness in combinatorics, see e.g. \cite{BrandenSurvey, VisontaiThesis}. 
Hyperbolicity and stability are generalizations of real rooted-ness for multivariate polynomials. 
A polynomial $f \in \C[x_1,\dots,x_n]$ is called {\em stable}  if $f(z) \neq 0$ for every point $z\in \C^n$ with $\im(z) \in \R_+^n$. 
Here $\im(z) = (\im(z_1), \dots, \im(z_n))$ denotes the imaginary part of $z$, and $\R_+$ denotes the set of positive real numbers.  
In other words, $f$ is stable if its roots do not lie in the ``upper half-plane'', 
and so stable polynomials are also called {\em polynomials with the half-plane property}.  
A univariate real polynomial is stable if and only if it is real-rooted.  More generally, a real polynomial $f$ is stable if and only if any line in any positive direction intersects the hypersurface $f = 0$ only at real points, that is, for any $w \in \R^n$ and $v \in \R_+^n$ the univariate polynomial $f(tv+w)\in \R[t]$ has only real roots.  

Another closely related generalization of real rooted-ness is \emph{hyperbolicity}, which can be thought of as a coordinate-free version of stability for real projective hypersurfaces.  
Concretely, a homogeneous polynomial $f\in \R[x_1,\dots,x_n]$ is called {\em hyperbolic} with respect to a point $v \in \R^n$ if $f(v)\neq 0$ and for any $w \in \R^n$ the univariate polynomial $f(tv+w)$ has only real roots.  Note that a homogeneous polynomial with real coefficients is stable if and only if it is hyperbolic with respect to every point $v \in \R_+^n$.  

The coordinate-dependent stability property imposes a rich structure on the coefficients of a stable polynomial with respect to the monomial basis. 
Choe, Oxley, Sokal, and Wagner showed that the Newton polytope of a homogeneous multiaffine stable polynomial must be a matroid polytope~\cite{COSW}.  Br\"and\'en generalized this result by showing that the Newton polytope of any homogeneous stable polynomial must be a {\em generalized permutohedron}, also known as an {\em M-convex polytope}.  These are the polytopes whose edges are in directions $e_i - e_j$, or equivalently, their normal fans coarsen the permutohedral fan.  Even more generally, Br\"and\'en showed that if $f$ is a stable polynomial over a real field with a non-Archimedean valuation, then the valuations of the coefficients of $f$ must form an $M$-concave function on its support, that is, all the faces of the induced regular subdivision must be $M$-convex~\cite{Branden2}. 

The notion of hyperbolicity was extended to projective varieties of codimension more than one by Shamovich and Vinnikov~\cite{SV}. 
Let $L \subseteq \P^{n}(\C)$ be a linear subspace defined over $\R$ of projective dimension $c-1$.
A  real projective variety $X \subseteq \P^{n}(\C)$ of codimension $c$ is said to be {\em hyperbolic with respect to $L$}
if  $X \cap L = \emptyset$ and for all real linear subspaces $L' \supset L$ of dimension $c$, the intersection $X \cap L'$ consists only of real points.
Just as the hyperbolicity of a polynomial $f$ can be certified by a determinantal representation $f = \det(A(x))$ where 
$A(x) = \sum_{i=1}^nx_iA_i$ where the matrices $A_i$ are real symmetric and $A(e)$ is positive definite, 
the hyperbolicity of a variety can be certified by a Livsic-type determinantal representation, which gives a definite determinantal representation of its Chow form~\cite{SV}. 
Hyperbolic varieties have also been studied in the context of \emph{real-fibered morphisms}, e.g. \cite{KS1, KS2, KShaw}. 
Kummer and Vinzant study the reciprocal linear space $L^{-1}$ of a real linear subspace $L$, which is hyperbolic with respect to $L^\perp$~\cite{KV}.

In this paper we study an analogue of stability called \emph{positive hyperbolicity} for varieties of 
codimension greater than one. For real projective varieties, this is equivalent to hyperbolicity 
with respect to all linear spaces in the positive Grassmannian. 

Our main goal is to explore the combinatorial structure of positively hyperbolic varieties through tropical geometry.  The tropicalization of a variety, defined in Section~\ref{sec:tropical}, is a polyhedral complex that sees various discrete invariants of the variety.  In particular, if $V$ is a hypersurface, the Newton polytope of the defining polynomial of $X$ can be recovered from the tropicalization of $X$.  For arbitrary varieties, tropicalization can be considered as a generalization of the notion of Newton polytope.

In Section~\ref{sec:background} we give a characterization of positively hyperbolic varieties in terms of sign variations.
They can also be characterized by their imaginary projections studied by J\"orgens, Theobald, and de Wolff in the hypersurface case~\cite{ImagProj}.
In Section~\ref{sec:preservers} we discuss different operations that preserve positive hyperbolicity, such as initial degenerations and certain linear transformations. 

\begin{figure}[ht]
\centering
\includegraphics[height=1.5in]{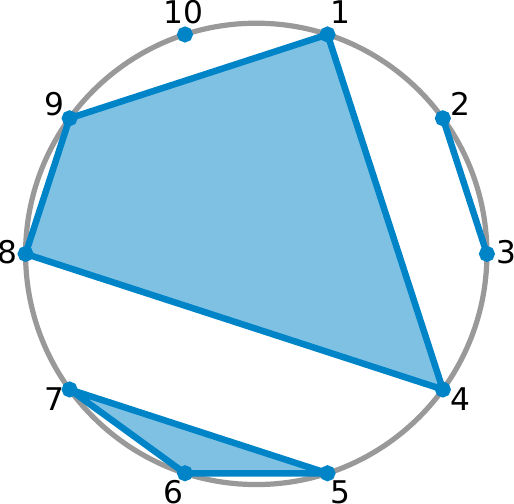}
\caption{The partition $\{\{1,4,8,9\},\{2,3\},\{5,6,7\},\{10\}\}$ is non-crossing.}
\label{fig:noncrossing}
\end{figure}

In Section~\ref{sec:tropical} we study tropicalizations of positively hyperbolic varieties. 
One of the key combinatorial structures that appear in this context 
is \emph{non-crossing partitions}. 
See Figure~\ref{fig:noncrossing} for an example and 
Section~\ref{sec:tropical} for a formal definition. 
We characterize toric varieties that are positively hyperbolic and we use this to obtain the following main result. 

\newtheorem*{thm:tropical}{Theorem \ref{thm:tropical}}
\begin{thm:tropical}
If $X \subset \mathbb C^{n}$ is a positively hyperbolic variety then the linear subspace parallel to any maximal face of $\trop(X)$ is spanned by $0/\pm 1$ vectors whose supports form a non-crossing partition of a subset of the set $[n] = \{1,\dots,n\}$. 
If, in addition, $X$ is homogeneous, then the linear subspace parallel to any maximal face must be spanned just by $0/1$ vectors whose supports form a non-crossing partition of $[n]$.
\end{thm:tropical}

\begin{figure}[ht]
\begin{center}
\includegraphics[height=2in]{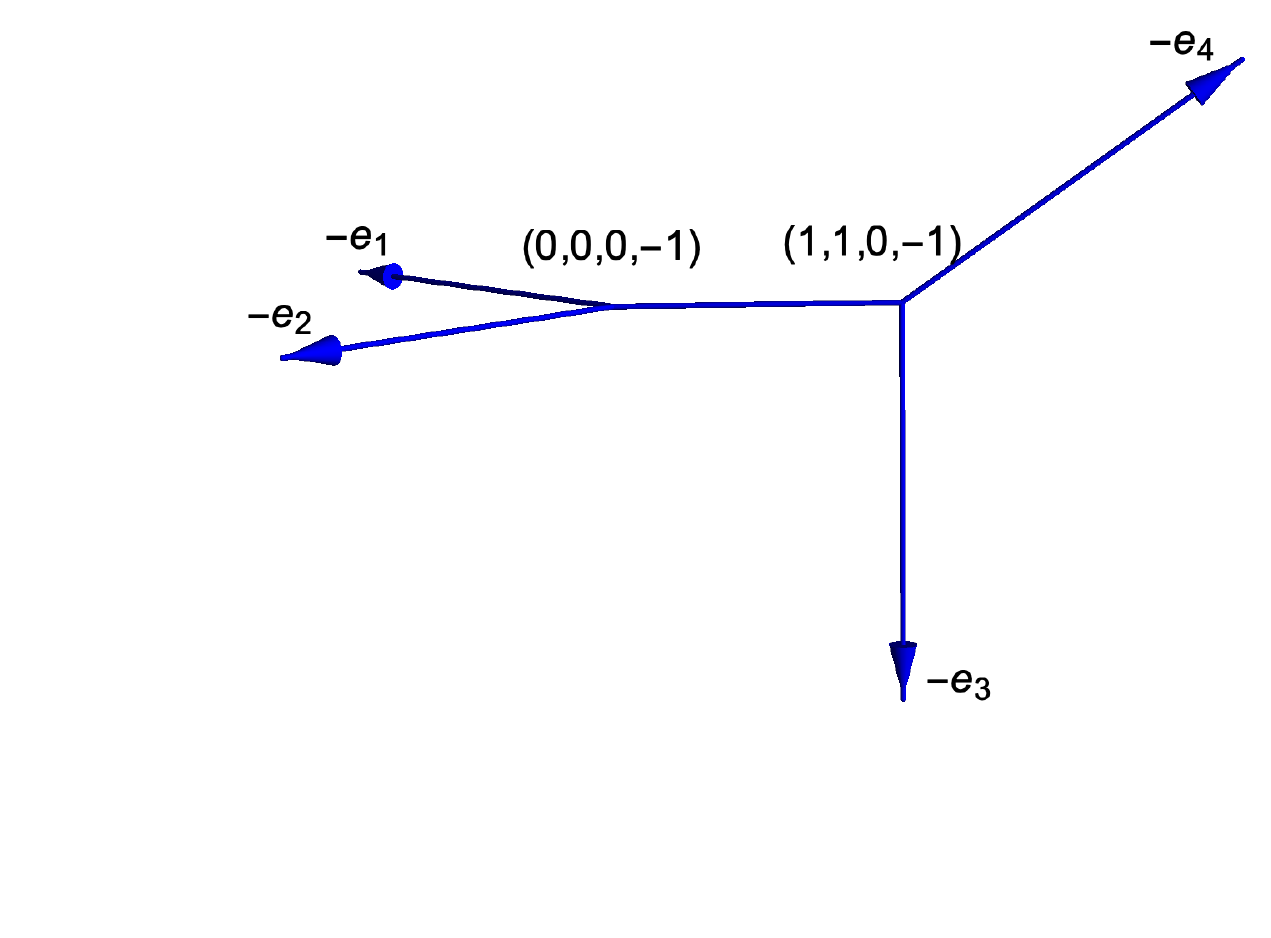}
\end{center}
\caption{The tropicalization of the reciprocal plane in Example~\ref{ex:reciprocalPlane}.} 
\label{fig:RecipLine}
\end{figure}

\begin{example}\label{ex:reciprocalPlane}
Let $\C\{\!\{t\}\!\}$ denote the field of Puiseux series. Consider the two-dimensional plane $\mathcal{L}$ in $\C\{\!\{t\}\!\}^4$ defined by 
\[
x_1+x_2+x_3+x_4 = 0 \ \ \text{ and }\  t^3x_1 +t^2x_2 + tx_3 + x_4 =0,
\]
and the reciprocal plane $\mathcal{L}^{-1}$ obtained by taking the image of $\mathcal{L}$ 
under the rational map $(x_1, x_2, x_3, x_4)\mapsto (x_1^{-1},x_2^{-1},x_3^{-1},x_4^{-1})$.
It follows from \cite[Cor. 1.5]{KV} that $\mathcal{L}^{-1}$ is a positively hyperbolic variety. This also follows from 
Proposition~\ref{prop:linearstable} and Corollary~\ref{cor:simplePreservers} below. 
The tropicalization of $\mathcal{L}^{-1}$ is a two-dimensional polyhedral complex in $\R^4$ 
with five maximal faces, all having the one-dimensional lineality space $\R(1,1,1,1)$. 
Its image in $\R^4/\R(1,1,1,1) \cong \R^3$ is shown in Figure~\ref{fig:RecipLine}. 
One maximal face is $\{(\lambda+\mu, \lambda+\mu, \lambda , \lambda  -1) : 0 \leq \lambda \text{ and } 0\leq \mu \leq 1\}$.
The linear subspace parallel to it is spanned by the vectors $(1,1,0,0)$ and $(0,0,1,1)$, which are 0/1 vectors 
whose supports form a non-crossing partition of $\{1,2,3,4\}$, namely $\{1,2\} \sqcup \{3,4\}$. 
One can check that the linear subspaces parallel to the other four maximal faces give rise to 
non-crossing partitions of the form $\{i\}\sqcup \{j,k,\ell\}$. Theorem~\ref{thm:tropical} prevents
the appearance of any cone in direction spanned by the vectors $(1,0,1,0)$ and $(0,1,0,1)$, as this would
correspond to the crossing partition $\{1,3\} \sqcup \{2,4\}$.
\end{example}

One immediate corollary is that if $X$ is a positively hyperbolic projective variety defined over a trivially valued field, then $\trop(X)$ is a subfan of the permutohedral fan given by the type $A$ braid arrangement $x_i = x_j$ for $i \neq j$.
The results of Choe--Oxley--Sokal--Wagner and Br\"and\'en mentioned above about the combinatorics of stable polynomials follow immediately from this result.  We also prove that the Chow polytopes of positively hyperbolic projective varieties are generalized permutohedra (i.e. $M$-convex polytopes).

While not all matroids appear as the support of stable polynomials (see \cite[Theorem~6.6]{Branden1}), there are three cases 
in which all polyhedral fans that have the correct combinatorial properties are actually realizable as tropicalizations of positively hyperbolic varieties.  
The first case is given by usual linear spaces, which are tropicalizations of toric varieties given by monomial parameterizations, as we will see in Section~\ref{sec:toric}.
The second is the case of fan tropical curves, corresponding to one-dimensional balanced polyhedral fans in $\R^n/\R(1,1,\dots,1)$, which we study in Section~\ref{sec:curves}. 
The last case is Bergman fans of matroids, corresponding to tropical varieties of degree one, discussed in Section~\ref{sec:Positroids}.

Given the special role of the positive Grassmannian, it is not surprising that the combinatorics of 
positively hyperbolic varieties is closely related to that of \emph{positroids}, which are matroids represented by linear subspaces in the nonnegative part of the Grassmannian.  We will show in Section~\ref{sec:algebraicmatroids} that the algebraic matroid of a positively hyperbolic variety is a positroid.
Ardila, Rinc\'on, and Williams showed in \cite{ARW1} that positroids are {\em non-crossing matroids}, that is, the ground sets of their connected components form a non-crossing partition. Moreover, they provided a characterization of positroids in terms of their matroid polytopes: A matroid is a positroid if and only if all the facets of its matroid polytope are non-crossing matroids.
In Proposition \ref{pro:positroidfacets} we extend their result to a characterization of positroids in terms of loopless faces of their matroid polytopes, and connect this to positive hyperbolicity.

\newtheorem*{thm:positroids}{Theorem \ref{thm:mainequivalencepositroids}}
\begin{thm:positroids}
Let $M$ be a loopless matroid on the set $[n] = \{1, \dots, n\}$. The following are equivalent statements:
\begin{enumerate}[label=(\alph*)]
\item The Bergman fan $\mathcal B(M)$ is the tropicalization of a positively hyperbolic variety.
\item $M$ is a positroid.
\item The linear subspace parallel to any maximal cone of $\mathcal B(M)$ is spanned by 0/1 vectors whose supports form a non-crossing partition.
\end{enumerate}
The word {\it maximal} can be removed in condition (c) above; see Remark \ref{rem:maximal}.
\end{thm:positroids}

\medskip
\noindent {\bf Acknowledgements.}  We thank Sergi Elizalde, Mario Kummer, Rainer Sinn, and Lauren Williams 
for helpful comments and discussions. This work started when FR and JY
were members of the Geometric and Topological Combinatorics program at
the Mathematical Sciences Research Institute and CV was a
member of the Discrete and Continuous Optimization program at the
Simons Institute in Berkeley during Fall 2017. It was later continued during the Fall 2018 Nonlinear Algebra program at the Institute for Computational and Experimental Research in Mathematics, in which the three authors took part.  FR was partially supported by the Research
Council of Norway grant 239968/F20.  
CV was supported by the US NSF-DMS grant \#1620014. 
JY was supported by the US
NSF-DMS grant \#1600569.

\section{Positively hyperbolic varieties}
\label{sec:background}

The \emph{positive Grassmannian} $\Gr_+(c,n)$ (resp.\ \emph{nonnegative Grassmannian} $\Gr_{\geq 0}(c,n)$) consists of $c$-dimensional linear subspaces of $\R^n$ all of whose Pl\"ucker coordinates are strictly positive (resp.\ nonnegative). In other words, a $c$-dimensional linear subspace $L \subset \mathbb R^n$ is called \emph{positive} (resp.\ \emph{nonnegative}) if it can be represented as the row space of a matrix in $\mathbb R^{c \times n}$ all of whose maximal minors are positive (resp.\ nonnegative).
Note that these notions depend heavily on the ordering of coordinates of $\R^n$.  

\begin{definition}
Let  $X\subset \C^n$ be a variety which is equidimensional of
codimension $c \leq n-1$. 
We call $X$ {\bf positively hyperbolic} if for every linear subspace $L$ in the positive Grassmannian $\Gr_+(c,n)$ 
and every $x \in X$, the imaginary part ${\rm Im}(x)$ does not belong to $L \backslash \{0\}$. We say that a projective variety in $\P^{n-1}$ is positively hyperbolic if its affine cone in $\C^n$ is.
\end{definition}

An \emph{equidimensional} variety is a variety whose irreducible components 
all have the same dimension.  It follows from the definition that an equidimensional variety $X$ 
is positively hyperbolic if and only if all of its irreducible components are. 

In \cite[\S 1]{KV}, projective positively hyperbolic varieties are considered 
and are called {\em stable} varieties.  
Here we avoid the terminology ``stable varieties'' both because of a
slight discrepancy in the definitions for affine hypersurfaces 
and because of the prevalence of the term ``stable'' in algebraic and tropical geometry.

When $c=1$, $X$ is a hypersurface, and the positive Grassmannian $\Gr_+(1,n)$ consists of lines spanned by vectors in the open positive orthant $\R_+^n$.  
In this case, positive hyperbolicity is almost the same as stability. 

\begin{proposition}
\label{prop:connectiontostable}
Let $f\in \C[x_1, \hdots, x_n]$. 
If the hypersurface $V(f)$ defined by $f=0$ is positively hyperbolic, then $f$ is stable. Furthermore, 
if $f$ is either homogeneous or has real coefficients, then $f$ is stable if and only $V(f)$ is positively hyperbolic. 
\end{proposition}
\begin{proof}
If ${\rm Im}(x)$ belongs to $\R_+^n$ for some $x\in V(f)$, then it is a nonzero point of the linear subspace $L = {\rm span}\{ {\rm Im}(x)\}$
in the positive Grassmannian $\Gr_+(1,n)$.  Therefore if $f$ is not stable, then $V(f)$ cannot be positively hyperbolic. 

If $f$ is either real or homogeneous, then $X=V(f)$ has the property that for any $x\in X$ there exists a point $y\in X$ such that ${\rm Im}(y) = -{\rm Im}(x)$, 
namely $y = \overline{x}$ if $f$ is real and $y = -x$ if $f$ is homogeneous.  
The nonzero points of the lines in $\Gr_+(1,n)$ are exactly those points in the positive and negative open orthants $\R_+^n\cup \R_-^n$. 
Therefore, there exists a point $x\in X$ with ${\rm Im}(x)\in L(\R)\backslash \{0\}$ for some $L\in \Gr_+(1,n)$ 
if and only if there is a point $x\in X$ with ${\rm Im}(x)\in \R_+^n$. 
\end{proof}

When $f$ is neither real nor homogeneous, positive hyperbolicity and
stability do not always coincide, because stability distinguishes
between the upper and lower half planes of the complex
numbers while positive hyperbolicity does not. 
For example, $x_1+x_2+i \in \C[x_1,x_2]$ is stable but is not
positively hyperbolic.  Its complex variety 
contains a point $(1-i/2, -1-i/2)$ whose imaginary part $(-1/2,-1/2)$ is contained in a linear subspace in $\Gr_+(1,2)$, namely $\C\{(1,1)\}$. 
It would be desirable to have a definition of positive hyperbolicity 
that generalizes stability in all cases, but we leave that for future work. 

\begin{remark}
For projective varieties one could also define positive hyperbolicity 
using real parts instead of imaginary parts, since $x\in X$ if and only if $i\cdot x\in X$. 
That is, a projective variety is positively hyperbolic if and only if for every point $x$ in its affine cone, 
${\rm Re}(x)\not\in L\backslash\{0\}$ for all $L\in \Gr_+(c,n)$. 
\end{remark}

The points in nonnegative/positive linear subspaces can be
characterized using sign variations, as shown by Gantmaher and Kre\u{\i}n~\cite{GK}.  For a vector $v = (v_1,
\dots, v_n) \in \R^n$, let $\var(v)$ be the number of sign changes in
the sequence $v_1, v_2, \dots, v_n$ after discarding any zeroes, and
let $\varbar(v)$ be the number of sign changes in the sequence $v_1,
v_2, \dots, v_n$ where the zeroes are assigned signs that maximize the
number of sign changes.  For example, $\var(1,0,0,1,-1)=1$ and  $\varbar(1,0,0,1,-1)=3$.
\begin{theorem}[\cite{GK, KWZ}]
\label{thm:GK}
  For $L \in \Gr(c,n)$, we have
  \begin{enumerate}
  \item $L \in \Gr_{\geq 0}(c,n) \Leftrightarrow \var(v) < c \text{ for all }v \in L \backslash\{0\} \Leftrightarrow \varbar(w) \geq c \text{ for all } w \in L^\perp \backslash\{0\}.$
  \item $L \in \Gr_{+}(c,n) \  \Leftrightarrow \varbar(v) < c \text{ for all }v \in L  \backslash\{0\} \Leftrightarrow \var(w) \geq c \text{ for all } w \in L^\perp \backslash\{0\}.$
  \end{enumerate}
\end{theorem}

Karp uses this to give the following characterization of points in positive linear subspaces: 

\begin{lemma}\cite[Lemma~4.1]{Karp}\label{lem:Karp}
  Let $v \in \R^n\backslash \{0\}$.
  \begin{enumerate}
  \item There exists a linear subspace in $\Gr_{\geq 0}(c,n)$ containing $v$ iff $\var(v) < c$.
    \item There exists a linear subspace in $\Gr_{+}(c,n)$ containing $v$ iff $\overline{\var}(v) < c$.
    \end{enumerate}
\end{lemma}
Note that, in particular, part (2) implies that the set of $v \in \R^n\setminus \{0\}$ that are contained in a linear subspace
$L \in \Gr_{+}(c,n)$ is an open subset of $\R^n$.

Lemma~\ref{lem:Karp} (2)  gives the following useful characterization of positive hyperbolicity.
\begin{proposition}
\label{prop:signs}
Let $X \subset \C^n$ be an equidimensional variety 
of codimension $c$.  Then $X$ is positively hyperbolic if and only if  $\varbar(\im(x)) \geq c$ for every $x \in X$.
\end{proposition}

For connections with tropical geometry, it will be important to consider 
points with nonzero coordinates.  
The following lemma states that this is sufficient in many cases. 

\begin{lemma}
\label{lem:stable} Let $X\subset \C^n$ be an irreducible variety of codimension $c \leq n-1$ 
not contained in any coordinate hyperplane.  Then $X$ is positively hyperbolic 
if and only if for all $L\in \Gr_+(c,n)$ and all $x\in X$ with $\im(x) \in (\R^*)^n$ we have ${\rm Im}(x)\not\in L$.
\end{lemma}

\begin{proof} The implication ($\Rightarrow$) is clear from definition. To prove ($\Leftarrow)$, suppose that $X$ is not positively hyperbolic. By definition, 
there is some point $x\in X$ with ${\rm Im}(x)\in L\backslash\{0\}$ for some $L\in \Gr_+(c,n)$.  If ${\rm Im}(x)\in (\R^*)^n$, we are done. Otherwise, 
consider a ball $B_\epsilon(x) \subset \C^n$ of radius $\epsilon$ around $x$. 
Since $X$ is not contained in any coordinate hyperplane, there exists 
a point $\tilde{x}$ in $B_\epsilon(x)\cap X$ whose imaginary part ${\rm Im}(\tilde{x})$ is in $(\R^*)^n$. 
For small enough $\epsilon$, $\tilde{x}$ is in a linear subspace $\tilde{L} \in \Gr_+(c,n)$, showing that $X$ is not positively hyperbolic.  
\end{proof}

We now characterize linear subspaces that are positively hyperbolic, using the following. 

\begin{lemma}\label{lem:sign}
For any integers $1\leq i_1 < \dots < i_k \leq n$, there is a choice of  
signs $\sigma\in \{\pm1\}^k$ such that for any point 
$z\in \R^n$ with ${\rm sign}(z_{i_j})=\sigma_j$, we have $\varbar(z) \leq n-k$. 
\end{lemma}
\begin{proof}
We can choose $\sigma_1 = 1$, and inductively 
choose the remaining signs so that $i_j$ and $i_{j+1}$ have
opposite signs if $i_{j+1} - i_j$ is even and the same sign if the
difference is odd. 
In total there are $n-1$ possible sign changes in a vector of length $n$. 
Between the $i_j$ and $i_{j+1}$ positions there are at most $i_{j+1} - i_j$ 
sign changes, but our choice of signs ensures that there are there are at most  $i_{j+1} - i_j -1$. 
This means that for any vector $z$ with the specified coordinate signs has 
at most $(n-1) - (k-1) = n-k$ sign changes in total, giving $\varbar(z) \leq n-k$.
\end{proof}

For example, for $n=7$, $k=4$, and the subset $\{1,2,4,7\}$, 
we can take the sign pattern $(++0-0\,0-)$. Any point of the form $z = (1,1,z_3,-1,z_5,z_6,-1)$ 
has  $\varbar(z)\leq 7-4=3$. 

\begin{proposition}\label{prop:linearstable}
A linear subspace $L \subset \C^n$  is positively hyperbolic if and only if it is
defined over $\R$ and the orthogonal complement of $L \cap \R^n$  is a nonnegative linear subspace. 
\end{proposition}

\begin{proof}
We first claim that a linear subspace $L \subset \C^n$ of complex dimension $d$ is defined over~$\R$ if and only if its imaginary part $\im(L) = \{\im(x) : x \in L\} \subset
\R^n$ has dimension $d$ as a real vector space. The ``only if'' direction of the claim is clear.  For the
``if'' direction, let us write $L$ as the row span of a matrix
$A+iB$ with $A, B \in \R^{d\times n}$. After row operations (and maybe column permutations) we can
assume that $A+iB$ starts with a $d \times d$ identity matrix. We can parametrize $L$ as the set of all  
$(x+iy)^T(A+iB)  = (x^TA-y^TB) + i(x^TB+y^TA)$ with $x, y \in \R^d$.
So the imaginary part $\im(L)$ is spanned by the rows of $A$ and
$B$. Suppose that $\im(L)$ is $d$-dimensional. Since $A$ starts with a $d
\times d$ identity matrix, this implies that the rows of $B$ lie in the row span of $A$. But the first $d$ entries of every row of $B$ are zero, which means that we must have $B=0$. 

We now show that if $L$ is positively hyperbolic then it must be defined over $\R$.
Indeed, if $L$ is not defined over $\R$, our claim implies that 
$\im(L)$ has dimension at least $d+1$. Then there is some choice of 
coordinates $i_1 < i_2 < \dots < i_{d+1}$ so that 
the projection of $\im(L)\rightarrow \R^{d+1}$ onto these coordinates is surjective. 
By Lemma~\ref{lem:sign}, we can choose values $z_{i_1},  \dots, z_{i_{d+1}} \in \R$
so that the corresponding vector $z = (z_1, \hdots, z_n)$ in $\im(L)$ has $\varbar(z) < n-d$.
This shows that $L$ is not positively hyperbolic.

We conclude the proof by noting that Proposition \ref{prop:signs} and part (1) of Theorem~\ref{thm:GK} imply that if $L$ is defined over $\R$ then $L$ is positively hyperbolic if and only if the orthogonal complement of $L\cap \R^n$ is a nonnegative linear subspace.
\end{proof}

The \emph{matroid of a linear subspace} $L$ encodes the linear
dependencies among the columns of any matrix $A$ such that $L$ 
is the row span of $A$.  
Matroids of nonnegative linear subspaces are called \emph{positroids};
that is, a matroid is a positroid if it can be represented
by the columns of a matrix with full row rank, all of whose maximal
minors are nonnegative.  Positroids were studied by Postnikov in order to understand topological properties of the non-negative Grassmannian \cite{Postnikov}, and they form a very special class of matroids with rich combinatorial properties.

\begin{corollary}
\label{cor:linearmatroid}
The matroid of a positively hyperbolic linear subspace is a positroid.
\end{corollary}

\begin{proof}
Let $L \subset \C^n$ be a positively hyperbolic linear subspace. 
By Proposition~\ref{prop:linearstable}, $L$ is defined over $\R$, and so the matroid of $L$ is the same
as the matroid of $L \cap \R^n$. The
matroid of $L \cap \R^n$ is dual to the matroid of its orthogonal complement
$(L\cap \R^n)^\perp$.  Since $(L\cap \R^n)^\perp$ is nonnegative by Proposition~\ref{prop:linearstable},
its matroid is a positroid. Ardila, Rinc\'on, and Williams showed that
the dual of a positroid is a positroid~\cite[Proposition 3.4]{ARW2},
and so the matroid of $L$ is a positroid.
\end{proof}

  This is already well known in the case $d=n-1$. 
A linear hyperplane $L$ is positively hyperbolic if and only if it 
 has a defining equation $\sum_{i=1}^na_i x_i=0$ where $a\in \R_{\geq 0}^n$, see e.g. \cite[Theorem 6.1]{COSW},
 \cite[Theorem 1.1]{Purbhoo}, \cite[Theorem 3.2]{ImagProj}.  
Positive hyperbolicity can also be characterized via imaginary projections, as we discuss in the next section.

\subsection{Connections of with previous notions}  
In this section, we describe how 
positive hyperbolicity relates to other previously defined notions, 
including imaginary projections and hyperbolicity of varieties.  

J\"orgens, Theobald, and de Wolff \cite{ImagProj} study the  \emph{imaginary projection}
of a variety $X\subset \C^n$, 
\[
{\rm Im}(X)  \ = \ \{{\rm Im}(x) : x\in X\} \ \subset \R^n.
\]
The definition of positive hyperbolicity depends only on points $\im(x)$ for $x\in X$, so 
we can define it in terms of the imaginary projection, namely  
$X$ is positively hyperbolic if and only if ${\rm Im}(X)$ has trivial intersection with the real points of any linear subspace in the positive Grassmannian. Alternatively, $X$ is positively hyperbolic if and only if ${\rm Im}(X)$ is contained in the orthants defined by the condition $\varbar(x) \geq c$. 

One of the main motivations for considering positive hyperbolicity 
is the notion of hyperbolicity for varieties discussed in the introduction, introduced by Shamovich and Vinnikov in \cite{SV}.
As we show below, the two notions are tightly related.

\begin{proposition}
 \label{prop:hyperbolic}
 Let $X \subset \C^n$ be a real variety of codimension $c$ and $\ol{X} \subset \P^n$ its 
projective closure.   Then 
 \begin{itemize}
 \item[(a)] $\ol{X}$ is positively hyperbolic iff $\ol{X}$ is hyperbolic w.r.t.~every $L\in \Gr_+(c,n+1)$. 
   \item[(b)] $X$ is positively hyperbolic iff $\ol{X}$ is hyperbolic w.r.t. $\{0\}\times L$ for every $L\in \Gr_+(c,n)$. 
 \end{itemize}
\end{proposition}

\begin{proof} By \cite[Prop. 1.3]{KV}, 
$\ol{X}$ is hyperbolic with respect to $L$ if 
and only if there is no point $b$ in
 $L(\R)\backslash\{0\}$ for which $[a+ib]\in \ol{X}$ for some $a \in \R^{n+1}$. 
This immediately shows (a). 

(b)
($\Rightarrow$) 
Suppose $\ol{X}$ is not hyperbolic with respect to $\{0\}\times L$ for some $L\in \Gr_+(c,n)$.
Again using \cite[Prop.~1.3]{KV}, this is equivalent to the existence of $y = (x_0, x) \in \C^{n+1}$
with $[y] \in \ol{X}$ and $\im(y)\in \{0\}\times
L(\R)\backslash\{0\}$.
If $x_0\neq 0$, then the point $x_0^{-1}x$ belongs to $X$, and since $\im(x_0) =0$,
$\im(x_0^{-1}x) = x_0^{-1}\im(x) \in L(\R)\setminus \{0\}$, showing that $X$ is not positively hyperbolic.
If $x_0=0$, then consider a small perturbation 
$[\epsilon : \tilde{x}]
\in \ol{X}$ where  $\epsilon$ is a small real number.  This is possible because
$\{ [x_0: x] \in \overline{X} : x_0 \in \R \setminus \{0\}\} = X$ is dense in
$\overline{X}$.
For sufficiently small $\epsilon$, 
$\im(\tilde{x})$ belongs to a positive
linear subspace $\tilde{L} \in \Gr_+(c,n)$. 
By replacing $y$ by
$[\epsilon : \tilde{x}]$ and $L$ by $\tilde{L}$, the problem reduces to
the case $x_0 \in \R \setminus \{0\}$ shown above.
 
($\Leftarrow$)  Let $X$ be hyperbolic with respect to all linear subspaces $\{0\}\times L$ where $L\in \Gr_+(c,n)$. 
Suppose there exists a point $x \in X$ with $\im(x) \in L(\R)$ for some $L\in \Gr_+(c,n)$.
Then the point $y = (1,x)$ satisfies $[y]\in \ol{X}$ and $\im(y)\in \{0\}\times L(\R)$. 
The hyperbolicity of $X$, along with \cite[Prop. 1.3]{KV}, implies that $\im(y)=0$,
which shows that $X$ is positively hyperbolic. 
\end{proof}

 \section{Operations preserving positive hyperbolicity}\label{sec:preservers}
 
Part of the rich theory of stable polynomials 
is a classification of linear operations preserving stability. 
Some basic examples are the following. 

\begin{lemma}[{\cite[Lemma 2.4]{wagnerSurvey}}] \label{lem:StablePreservers}
The following operations on $\C[x_1, \hdots, x_n]$ 
preserve stability.
\begin{itemize}
\item Permutation: $f\mapsto f(x_{\pi(1)}, \hdots, x_{\pi(n)})$ for any $\pi \in S_n$
\item Scaling:  $f\mapsto c f(a_1 x_1, \hdots, a_nx_n)$ for any $c\in \R$, $a_1, \hdots, a_n\in \R_{\geq 0}$
\item Diagonalization: $f\mapsto f(x_1, x_1, x_3, \hdots, x_n)$
\item Specialization:  $f\mapsto f(a, x_2, x_3, \hdots, x_n)$ for $a\in \R_{\geq 0}$
\item Inversion:  $f\mapsto x_1^{d} f(-x_1^{-1}, x_2, x_3, \hdots, x_n)$ where $f$ has degree $d$ in $x_1$
\item Differentiation:  $f\mapsto \partial f/\partial x_1$
\end{itemize}
\end{lemma}

Borcea and Br\"and\'en gave a full characterization of linear operations preserving stability \cite{BB}. 
See also Section 5 of \cite{wagnerSurvey}. 
Some of these operations extend immediately to preservers of positive
hyperbolicity for varieties,
but the analogues of others are less clear.

\begin{proposition}\label{prop:linearPreserver}
Let $T:\C^n \to \C^m$ be a surjective linear map defined over the real numbers.
For $c\leq m$, the following are equivalent: 
\begin{itemize}
\item[(a)] For all equidimensional positively hyperbolic varieties $X \subset \C^m$ of codimension $c$, $T^{-1}(X)$ is positively hyperbolic.
\item[(b)] For all $x \in \R^n\setminus \{0\}$ with $\varbar(x) < c$, we have $\varbar(T(x)) < c$.
\item[(c)] For all $L \in \Gr_+(c,n)$ we have $T(L) \in \Gr_+(c,m)$.
\item[(d)] All of the nonzero $c \times c$ minors of the matrix representing $T$ have the same sign.
\end{itemize}
\end{proposition}

\begin{proof}
(d $\Rightarrow$ c) Let $M$ denote the $m\times n$ matrix representing the linear map $T$.
We can write any linear subspace $L \subset \R^n$ of dimension $c$ 
as the column span of a $n\times c$ matrix $B$. Then $T(L)$ is the column span of the $m\times c$ matrix $M\cdot B$. By the Cauchy-Binet formula, 
for any $I\in \binom{[m]}{c}$, the maximal minor of $M\cdot B$ corresponding to $I$ is given by 
\[
\det((M\cdot B)_I) \ = \ \sum_{J\in \binom{[n]}{c}} \det(M_{I,J})\cdot \det(B_{J}).
\]
If $L\in \Gr_+(c,n)$, then we can take $B$ such that all its maximal minors are positive. If all of the nonzero $c\times c$ minors of $M$ 
have the same sign, then from the formula above, 
all of the nonzero $c\times c$ minors of $M\cdot B$ also have this sign.  
Moreover, if some minor $\det((M\cdot B)_I) $ equals zero
then $\det(M_{I,J})$ must be zero for all $J\in \binom{[n]}{c}$, meaning that the $c\times n$ matrix $M_{I,\cdot}$ has rank $<c$. 
Note that this matrix represents the composition of the linear map $T$ with the projection $\pi_I:\C^m\rightarrow \C^I$. The surjectivity of $T$ implies that 
this composition is also surjective, meaning that the matrix $M_{I,\cdot}$ has full rank $c$, giving a contradiction. 
All together this shows that $M\cdot B$ has full rank $c$ with all nonzero minors of the same sign. Therefore $T(L)\in \Gr_+(c,n)$.

(c $\Rightarrow$ b) 
Let $x\in \R^n$ with $\varbar(x)<c$. By Lemma~\ref{lem:Karp}, there exists a 
linear space $L\in \Gr_+(c,n)$ with $x\in L$.  
Then $T(L)$ belongs to the positive Grassmannian $\Gr_+(c,m)$. 
This implies that $T$ is injective on $L$, and so $T(x)$ is an element of $T(L)\backslash\{0\}$. 
Using this lemma again we get that $\varbar(T(x))<c$. 

(b $\Rightarrow$ a) 
Note that because $T$ is surjective and $X$ has codimension $c$ in $\C^m$,
$T^{-1}(X)$ is a variety of codimension $c$ in $\C^n$. Suppose that $T^{-1}(X)$ is not positively hyperbolic, 
meaning that there exists a point $y\in T^{-1}(X)$ with
$\varbar(\im(y))<c$.  Consider the point $x = T(y) \in X$.  
By linearity, $\im(x)$ equals $T(\im(y))$, which by assumption satisfies $\varbar(\im(x))<c$,
contradicting the 
positive hyperbolicity of $X$.

(a $\Rightarrow$ d) 
For $L\in \Gr_{+}(c,m)$, consider the linear variety $X=L^{\perp}$ defined by 
$\ell\cdot x=0$ for $\ell\in L$. 
By Proposition~\ref{prop:linearstable},  $X$ is positively hyperbolic and therefore so is $T^{-1}(X)$. 
The variety 
$T^{-1}(X)$ is also a linear variety, defined by $\ell\cdot T(y) = \ell \cdot M y=0$. 
If $L$ is the row span of a $c\times m$ matrix $A$, then the orthogonal complement of $T^{-1}(X) \cap \R^n$
is the row span of the $c\times n$ matrix $A\cdot M$. 
By Proposition~\ref{prop:linearstable},  this linear space belongs to $\Gr_{\geq 0}(c,n)$, 
and so the $c\times c$ minors of $A\cdot M$ all have the same sign. 

Without loss of generality, we may assume that all the $c\times c$ minors of $A\cdot M$ are nonnegative, 
and suppose for the sake of contradiction that some $c\times c$ minor of $M$ is negative, say 
$\det(M_{I,J})<0$ for some $I\in \binom{[m]}{c}$ and $J\in \binom{[n]}{c}$. 
Let $\lambda^I\cdot A$ denote the matrix obtained from scaling the $i$th column of 
$A$ by $\lambda\in \R_+$ for all $i\in I$.  
The resulting matrix $\lambda^I\cdot A$ 
still has positive minors and so the $c\times c$ minors of  $\lambda^I \cdot A\cdot M$ all have the same sign. 
For $\lambda=1$ they are positive. As $\lambda \to \infty$, the $J$th minor of $\lambda^I \cdot A\cdot M$ must become negative. 
This implies that there is some $\lambda\in (1,\infty)$ at which all the $c\times c$ minors of   
$\lambda^I\cdot  A \cdot M$ change sign. 
In particular all the $c \times c$ minors must be zero, meaning that the matrix drops rank.  
However this contradicts the surjectivity of the composition of $T$ with the 
linear map $\C^m\to \C^c$ defined by $\lambda^I\cdot A$.
Therefore every $c\times c$ minor of $M$ must have the same sign. 
\end{proof}

\begin{corollary}\label{cor:simplePreservers}
The following operations preserve positive hyperbolicity of equidimensional varieties of codimension $c$ in $\C^n$: 
\begin{itemize}
\item Scaling:  $X\mapsto T(X)$ where $T(x_1, \hdots, x_n) = (a_1x_1, \hdots, a_nx_{n})$ where $a_i \in \R_{>0}$
\item Cyclic permutation: $X\mapsto {\rm cyc}_c(X)$ where ${\rm cyc}_c(x_1, \hdots, x_n) = ((-1)^{c-1} x_n, x_1, \hdots, x_{n-1})$
\item  Reversal: $X\mapsto {\rm rev}(X)$ where ${\rm rev}(x_1, \hdots, x_n) = (x_n, x_{n-1}, \hdots, x_{1})$
\item  Negation: $X\mapsto -{\rm id}(X)$ where $-{\rm id}(x_1, \hdots, x_n) = (-x_1, -x_2, \hdots, -x_n)$
\item Inversion:  $X\mapsto \overline{T(X)}^{Zar}$ where $T(x_1, \hdots, x_n) =  (x_1, \hdots,x_{i-1}, -x_{i}^{-1}, x_{i+1}, \hdots, x_n)$. 
\end{itemize}
\end{corollary}

\begin{proof} The first four follow from Proposition~\ref{prop:linearPreserver}. 
For the last, we note that for any nonzero complex number $z$, 
$\im(z)$ and $\im(-z^{-1})$ 
have the same sign. 
\end{proof}

We now characterize all signed permutations that preserve positive hyperbolicity.
Let $B_n$ denote the hyperoctahedral group (or signed symmetric group) consisting of all $2^n \, n!$ signed coordinate permutations of $\C^n$, i.e.,
all maps $\phi: \C^n \to \C^n$ of the form $\phi(x_1, \dots, x_n) = (\pm x_{\sigma(1)}, \dots, \pm x_{\sigma(n)})$ for some
permutation $\sigma$ of $[n]$. 
The signed permutations with all signs positive form a copy of the symmetric group $S_n \leq B_n$. 
In the next proposition we refer to some particular elements of $B_n$,
namely the operations ${\rm cyc}_c$, ${\rm rev}$, and $-{\rm id}$ defined in Corollary~\ref{cor:simplePreservers}.

\begin{proposition}
Fix $1 \leq c \leq n-1$, and let $G \leq B_n$ be the subgroup consisting of all signed coordinate permutations that preserve positive hyperbolicity of equidimensional varieties of codimension $c$ in $\C^n$.
We have:
\begin{enumerate}
\item If $c=1$ then $G = \langle S_n, -{\rm id} \rangle \cong S_n \times \mathbb Z/2\mathbb Z$.
\item If $2 \leq c \leq n-2$ then $G = \langle {\rm cyc}_c, {\rm rev}, -{\rm id} \rangle$. If $c$ is odd then $G \cong D_{2n} \times \mathbb Z/2\mathbb Z$,
while if $c$ is even then $G$ is isomorphic to $D_{4n}$. 
\item If $c = n-1$ then $G$ is isomorphic to $S_n \times \mathbb Z/2\mathbb Z$.
\end{enumerate}
\end{proposition}
\begin{proof}
If $c=1$ then any $\phi \in \langle S_n, -{\rm id} \rangle$ preserves the property $\overline{\var}(y) < 1$, and so by Proposition \ref{prop:linearPreserver}, $\phi$ preserves positive hyperbolicity of hypersurfaces. Now, suppose $\phi \in G$. we can compose $\phi$ with an element of $S_n$ to get a $\phi'$ of the form $\phi'(x_1,\dots,x_n)=(\pm x_{1}, \dots, \pm x_{n})$ that also preserves positive hyperbolicity. Since $\phi'$ must preserve the property $\overline{\var}(y) < 1$, it follows that $\phi'$ must be equal to $\pm {\rm id}$. This shows that $G = \langle S_n, -{\rm id} \rangle$, as claimed.

Suppose now that $2 \leq c \leq n-2$. As discussed in Corollary
\ref{cor:simplePreservers}, the three signed permutations ${\rm
  cyc}_c$, ${\rm rev}$, and $-{\rm id}$ preserve positive
hyperbolicity. Suppose now that $\phi \in G$. By Proposition
\ref{prop:linearPreserver}, $\phi$ preserves the positive Grassmannian
$\Gr_+(c,n)$, and thus by continuity, $\phi$ also preserves the
nonnegative Grassmannian $\Gr_{\geq 0}(c,n)$. The matroid associated
to any nonnegative $c$-dimensional linear subspace has connected
components that form a non-crossing partition of $[n]$ with at most
$c$ parts. (See Definition~\ref{def:noncrossing} for a definition.)
Moreover, every non-crossing partition of $[n]$  with $c$ parts arises this way  \cite[Theorem 7.6]{ARW1}. 
It follows that, if $\phi$ is given by $\phi(x_1,\dots,x_n)=(\pm x_{\sigma(1)}, \dots, \pm x_{\sigma(n)})$ for some permutation $\sigma$ of $[n]$, then $\sigma$ must send any non-crossing partition of $[n]$ with $c$ parts into a non-crossing partition. 

We claim that $\sigma$ must be a permutation in the dihedral group $D_{2n}$, i.e., $\sigma$ must send any two cyclically consecutive elements of $[n]$ to cyclically consecutive elements of $[n]$. Suppose not, and let $i, i+1$ be two cyclically consecutive elements of $[n]$ such that $\sigma(i), \sigma(i+1)$ are not cyclically consecutive (where $n+1$ denotes the element $1$). Then there exist $j, k \in [n]$ such that $j$ is in the cyclic interval $(\sigma(i), \sigma(i+1))$ and $k$ is in the cyclic interval $(\sigma(i+1), \sigma(i))$. Since $2 \leq c \leq n-2$, we can find a non-crossing partition $\mathcal P$ of $[n]$ with $c$ parts such that $i$ and $i+1$ are in the same part of $\mathcal P$, and also $\phi^{-1}(j)$ and $\phi^{-1}(k)$ are in the same part of $\mathcal P$.  (See Definition~\ref{def:noncrossing} for the formal definition of non-crossing.)
But then $\sigma$ sends the non-crossing partition $\mathcal P$ into a crossing partition, which is a contradiction. 
%

Now, note that ${\rm cyc}_c, {\rm rev} \in B_n$ generate a subgroup isomorphic to the dihedral group $D_{2n}$ of $2n$ elements. Since $\sigma$ is a permutation of $[n]$ lying in the dihedral group $D_{2n}$, we can compose $\phi$ with an element of the subgroup $\langle {\rm cyc}_c, {\rm rev}\rangle$ in order to get a $\phi' \in G$ of the form $\phi'(x_1,\dots,x_n)=(\pm x_{1}, \dots, \pm x_{n})$. Denote by $S \subset [n]$ the set of coordinates $i$ such that $\phi'(x)_i = -x_i$. Since $\phi'$ preserves nonnegativity of $c$-dimensional linear subspaces, we must have that the parity of $|S \cap T|$ for $T \in \binom{[n]}{c}$ does not depend on $T$. This implies that $S = \emptyset$ or $S = [n]$, and thus $\phi' = \pm {\rm id}$, showing that $\phi \in \langle {\rm cyc}_c, {\rm rev}, -{\rm id} \rangle$, as claimed.

The structure of the subgroup $G = \langle {\rm cyc}_c, {\rm rev}, -{\rm id} \rangle$ depends on the parity of $c$. If $c$ is odd then ${\rm cyc}_c^n = {\rm id}$, and $G$ is isomorphic to $\langle {\rm cyc}_c, {\rm rev} \rangle \times \langle -{\rm id} \rangle \cong D_{2n} \times \mathbb Z/2\mathbb Z$. If $c$ is even then ${\rm cyc}_c^n = -{\rm id}$, and so ${\rm cyc}_c$ is an element of order $2n$. The subgroup $G$ is in this case isomorphic to $D_{4n}$, as it has the presentation $G = \langle {\rm cyc}_c, {\rm rev} \mid {\rm cyc}_c^{2n} = {\rm rev}^2 = ({\rm cyc}_c \cdot {\rm rev})^2 = {\rm id} \rangle$.

Finally, suppose that $c = n-1$. Note that if $\overline{\var}(y) \geq n-1$ for $y \in (\mathbb R^*)^n$ then the signs in the sequence $y_1, y_2, \dots, y_n$ alternate at every step. For any permutation $\sigma$ of $[n]$ it is possible to choose signs such that the map $\phi \in B_n$ sending $(x_1,\dots,x_n)$ to $(\pm x_{\sigma(1)}, \dots, \pm x_{\sigma(n)})$ preserves this alternation at every step. Indeed, one possible choice of signs is as follows: if $\sigma$ sends $i$ to $j$ then we can declare $\phi(x)_i = -x_j$ exactly when $i$ and $j$ have different parity. The opposite choice of signs also preserves the sign alternation, and these are the only two possible choices. This shows that the set of $\phi \in B_n$ preserving the property $\overline{\var}(y) \geq n-1$ is equal to this subgroup isomorphic to $S_n \times \mathbb Z/2 \mathbb Z$. By Proposition \ref{prop:linearPreserver}, this is the subgroup $G$ of signed coordinate permutations that preserve positive hyperbolicity, completing the proof.
\end{proof}

Understanding linear maps $T$ that preserve positive hyperbolicity when taking the image of a variety (rather than its preimage) seems to be more subtle, as $T(X)$ might have a different codimension than $X$. 
\begin{question} What linear maps $T:\C^n \rightarrow \C^m$ have the 
property that if $X$ is positively hyperbolic, then so is $T(X)$?  What about rational maps? 
\end{question}

We note, though, that projecting onto consecutive coordinates does preserve positive hyperbolicity.

\begin{lemma}[Consecutive coordinate projections]\label{lem:stableProj}
 Suppose $X \subset \C^n$ is a positively hyperbolic variety.  For any subset $S \subset [n]$ consisting of consecutive elements, the (Zariski closure of the) projection $\pi_S(X)$ of $X$ onto the coordinates indexed by $S$ is also positively hyperbolic.
\end{lemma}

\begin{proof}
Let $X$ be a variety of dimension $d$.
We can assume that $X$ irreducible. Consider the projection $\pi_{[n-1]}: \C^n \rightarrow \C^{n-1}$ 
onto the first $n-1$ coordinates, and let $Y$ denote the projection 
$\pi_{[n-1]}(X)$ of $X$.
The dimension of $Y$ is either $d$ or $d-1$.
If $\dim(Y) = d$, then for any point $y = \pi_{[n-1]}(x) \in Y$,
\[
\varbar({\rm Im}(y))  \ = \  \varbar(\pi_{[n-1]}({\rm Im}(x))) 
\ \geq \ \varbar({\rm Im}(x)) - 1 
\  \geq \ \codim(X) - 1 
\ = \ \codim(Y),
\]
which shows that $Y$ is positively hyperbolic. 

If $\dim(Y) = d-1$, then 
for any $y\in Y$ and any $a \in \C$, 
the point $(y,a)$ belongs to $X$.  In particular, for any $y\in Y$ there is a point $x\in X$ 
with $\pi_{[n-1]}(x) = y$ and $\varbar({\rm Im}(x)) = \varbar({\rm Im}(y))$.  Then 
\[
\varbar({\rm Im}(y))  \ = \  \varbar({\rm Im}(x)) 
\  \geq \ \codim(X)  
\ = \ \codim(Y),
\]
giving again that $Y$ is positively hyperbolic. 

A similar argument shows that the projection onto the last $n-1$ coordinates also preserves positive hyperbolicity.  By iterating these projections, we conclude that the projection onto any set of consecutive coordinates is positively hyperbolic.
\end{proof}

\begin{lemma}[Coordinate projections with sign changes]\label{lem:projectSign}
Suppose $X \subset \C^n$ is a positively hyperbolic variety.  For any
subset $S \subset [n]$, the projection $\pi_S(X)$ is also positively
hyperbolic after possibly changing signs of some coordinates.
\end{lemma}

\begin{proof}
We can project out any one coordinate after preforming a suitable cyclic shift, followed by a projection onto the first $n-1$ coordinates. The statement then follows from the previous two lemmas by doing this repeatedly.
\end{proof}
 
  \begin{lemma}[Embedding in a coordinate subspace]
Suppose $X\subset \C^n$ is a variety contained in $\C^S = \{x \in \C^n : x_i = 0,  i\in [n]\backslash S\}$
for some consecutive subset $S \subset [n]$. Then $X$ is positively hyperbolic if and only if $\pi_S(X)$ is positively hyperbolic. 
 \end{lemma}

 \begin{proof}
The ``only if'' direction follows from Lemma~\ref{lem:stableProj}.
For the ``if'' direction, it suffices to show this for $S = [n-1]$. The general case follows by induction. 
Suppose $X \subset \C^{[n-1]}$ and $Y = \pi_{[n-1]}(X)$ is positively hyperbolic. Any point $x\in X$ has the form $(y,0)$ for some $y\in Y$,
and $\varbar(x) = \varbar(y)+1$.  Moreover, $\codim(X) = \codim(Y) + 1$.  Together, this gives 
\[\varbar(x) \ = \ \varbar(y) +1 \ \geq \ \codim(Y) + 1 = \codim(X),\]
and so $X$ is positively hyperbolic as well. 
\end{proof}

As in Lemma~\ref{lem:projectSign}, an analogous statement holds for
any (not necessarily consecutive) subset $S \subset [n]$, after possibly
changing signs of some coordinates.

\begin{lemma}[Product]
\label{lem:product}
If varieties $X_1\subset \C^{n_1}$ and $X_2\subset \C^{n_2}$ are both positively hyperbolic, then so is their product $X_1\times X_2 = \{(x_1,x_2) : x_1\in X_1, x_2\in X_2\} \subset \C^{n_1+n_2}$.
\end{lemma}

\begin{proof}
Let $c_i$ be the codimension of $X_i$ in $\C^{n_i}$. Then the variety $X_1\times X_2$ has codimension $c_1+c_2$ in $\C^{n_1+n_2}$.
For any point $(x_1,x_2)\in X_1\times X_2$, $\varbar(\im(x_1))\geq c_1$ and $\varbar(\im(x_2))\geq c_2$, so 
 $\varbar(\im(x_1,x_2)) \geq c_1+c_2$. 
\end{proof}

\subsection{Tangent spaces and algebraic matroids}
\label{sec:algebraicmatroids}
We now show that taking tangent spaces of positively hyperbolic varieties
results in positively hyperbolic linear subspaces.

\begin{proposition}\label{prop:tangent}
Let $X\subset \P^{n-1}(\C)$ be an irreducible projective variety defined over $\R$. 
If $X$ is  hyperbolic with respect to a linear subspace $L$, 
then so is the tangent space $T_pX$ of $X$ at any real smooth point $p\in X(\R)$.
\end{proposition}
\begin{proof}
Since $T_pX$ is a linear subspace, it is hyperbolic with respect to $L$ if and only if $T_pX\cap L = \emptyset$. 
Consider the map $\pi_L: X\rightarrow  \P(\C^n/L)\cong \P(L^{\perp})$ given by projection away from $L$. 
Since $X\cap L$ is empty and $\dim(X) = \dim( \P(L^{\perp}))$, this is a finite-to-one morphism. 
Moreover, the hyperbolicity of $X$ implies that for any real point in $\P(L^{\perp})$, the 
fiber is contained in $X(\R)$. 

By \cite[Theorem 2.19]{KS1}, for any smooth point $p$ of $X$, the differential map 
on tangent spaces $d_p\pi_L:T_pX\to  \P(L^{\perp})$ is surjective. Since this is a linear map 
of (projective) linear subspaces of the same dimension, it follows that it is also injective, meaning that $T_pX\cap L =\emptyset$. 
\end{proof}

Together with Proposition~\ref{prop:hyperbolic}, this gives the following. 

\begin{corollary}[Tangent spaces]
Let $X\subset \C^n$ be an irreducible variety defined over $\R$. 
If $X$ is positively hyperbolic, then so is the tangent space $T_pX$ of $X$ at any real smooth point $p\in X(\R)$.
\end{corollary}

The {\em algebraic matroid} of an irreducible algebraic variety $X$ in
$K^n$ is a matroid on the ground set $[n]$, in which a set $S \subset [n]$ is
independent if and only if the projection of $X$ onto the coordinates indexed by $S$ is
dominant, that is, the image has full dimension $|S|$. If the ground field
$K$ has characteristic $0$, then the algebraic matroid of $X$
coincides with that of its tangent space at a general point $p \in
X$.  The following statement follows from Corollary~\ref{cor:linearmatroid}.

\begin{corollary}
\label{cor:algmatroid}
Let $X\subset \C^n$ be an irreducible variety defined over $\R$. 
If $X$ is positively hyperbolic, then the algebraic matroid of $X$ is a positroid.
\end{corollary}

\section{Initial Degeneration and Tropicalization}
\label{sec:tropical}

Let $K =\C\{\!\{t\}\!\} =  \bigcup_{k \geq 0} \C((t^\frac{1}{k}))$ be the field of Puiseux series with complex coefficients and let $S = K[x_1,\dots,x_n]$.  Let $A \subset K$ be the valuation ring of $K$, consisting of those series with valuation $\geq 0$, and let $\m \subset A$ be its unique maximal ideal, consisting of those series with valuation $>0$. Note that $A/\m = \C$.
For an ideal $I \subset S$ and a weight vector $w \in \R^n$, the {\em $t$-initial ideal} $\tinit_w(I) \subset \C[x_1,\dots,x_n]$ is defined as follows.  
For every nonzero $f \in I$, let $\tilde{f} \in A[x_1,\dots,x_n]$ be defined as
$\tilde{f} = t^{\mu} f(t^{w_1}x_1, \dots, t^{w_n}x_n)$,
where $\mu \in \R$ is chosen so that 
the smallest valuation is $0$ among the coefficients of $\tilde{f}$. 
The \emph{$t$-initial form} $\tinit_w(f)$ of $f$ is defined as the image of $\tilde{f}$ modulo $\m$, which lives in $\C[x_1,\dots,x_n]$.
The $t$-initial ideal of $I$ is defined as the ideal generated by $\{\tinit_w(f) : f\in I\}$.
We can also see this as the image of the ideal $\tilde{I} = \langle \tilde{f} : f\in I\rangle$ 
modulo the ideal $\m$.
For an ideal $J \subset \C[x_1,\dots,x_n]$ and $w \in \R^n$ we have
$\init_w(J) = \tinit_w(JS)$, so this is a generalization of the usual
initial ideals.  See also~\cite{JMM},
\cite[Section~2.4]{MaclaganSturmfels}, and
\cite[Chapter~15]{Eisenbud}.\footnote{Instead of $\C\{\!\{t\}\!\}$, we
  can start with an arbitrary real closed field $R$ with a compatible
  non-trivial non-archimedean valuation, and let $K = R[\sqrt{-1}]$,
  as in~\cite{JSY} for example. And we can define $t$-initial ideals and tropicalizations as in~\cite[Section~2.4]{MaclaganSturmfels}. But we will use Puiseux series here for concreteness.}

We will now show that positive hyperbolicity is preserved under taking $t$-initial ideals.
Note that we can write any element $z\in K$ as $a+ib$ where $a,b\in \R\{\!\{ t\}\!\}$ are 
real Puiseux series, and write $\im(z) = b$. 
Real Puiseux series form a real closed field, where two series $a,b \in \R\{\!\{ t\}\!\}$ satisfy $a < b$ if and only if $b-a$ has a positive leading coefficient. 
The definitions of positivity, the positive Grassmannian, and positive hyperbolicity of varieties extend verbatim. 

\begin{proposition}
  \label{prop:initial}
If the variety of $I \subset K[x_1,\dots,x_n]$ in $K^n$ is positively hyperbolic, then for any $w \in \R^n$, the variety of $\tinit_w(I)$ in $\C^n$ is also positively hyperbolic.
\end{proposition}

\begin{proof}
If $V(I)$ is an equidimensional variety of codimension $c$ then so is $V(\tinit_w(I))$~\cite[Corollary 6.17]{JMM}. 
It follows from the definitions above that
  \[
\C[x_1,\dots,x_n]/{\tinit_w(I)} = A[x_1,\dots,x_n]/(\tilde{I}+\m) = A/\m \otimes_A A[x_1,\dots,x_n]/{\tilde{I}}.
  \]
See \cite[Theorem~15.17]{Eisenbud}.  This means that points in $V_\C(\tinit_w(I))$ are obtained by taking the points of $V_K(\tilde{I}) \cap A^n$ and reducing them modulo $\m$, which does not decrease the sign variation $\varbar$.  The result then follows from Proposition~\ref{prop:signs}. 
\end{proof}

For an affine variety $X \subset K^n$ defined by an ideal $I$, the \emph{tropicalization} of $X$ is 
\[
\trop(X) = \{ w \in \R^n : \tinit_w(I) \text{  does not contain a monomial}\}.
\] 
When the variety is defined over a (sub)field with trivial valuation
such as $\C$, its tropicalization is a polyhedral fan which is a
subfan of the Gr\"obner fan.  This case is called the {\em constant
  coefficient case} in tropical geometry, as the coefficients of the
defining polynomials have constant valuation.  In general the
tropicalization is a polyhedral complex which is a subcomplex of the
Gr\"obner complex of the defining ideal~\cite{MaclaganSturmfels}.
It satisfies the {\em balancing condition} where each maximal cone has weight equal to the sum of multiplicities of the monomial-free
associated primes of the corresponding $t$-initial ideal.  

Our main goal is to explore the combinatorial structure of the
tropicalizations of positively hyperbolic varieties.  When the variety
$X$ is a hypersurface defined by a polynomial $f = 0$ over $\C$, then the
tropical variety $X$ is the union of normal cones to edges of the
Newton polytope of~$f$, which is the convex hull of exponent vectors
of monomials in $f$.  In particular, we can recover the edge
directions of the Newton polytope from the tropicalization, and the
weights give the edge lengths.

\begin{example}
\label{ex:reciprocalcircuits}
Consider the variety $\mathcal{L}^{-1}$ defined in Example~\ref{ex:reciprocalPlane}. 
From the linear equation $(1-t^3)x_1 + (1-t^2)x_2+(1-t)x_3=0$ vanishing on $\mathcal{L}$, we obtain a 
polynomial 
\[f_{123} = (1-t^3)x_2x_3 + (1-t^2)x_1x_3+(1-t)x_1x_2\] 
vanishing on $\mathcal{L}^{-1}$.
The collection of ``circuit'' polynomials $\{f_C : C \in \binom{[4]}{3}\}$ form a universal Gr\"obner basis 
for the ideal of polynomials vanishing on $\mathcal{L}^{-1}$ \cite{ProudfootSpeyer}. 
Let $w = (1/2,1/2,0,-1)$.  Then the $t$-initial form of $f_{123}$ is $\tinit_w(f) = x_2x_3+x_1x_3$. 
The $t$-initial ideal of $I  = I(\mathcal{L}^{-1})$ is generated by the $t$-initial forms of the circuit polynomials:
\begin{align*}
\tinit_w(I)  & = \langle x_3(x_1+x_2),  x_4(x_1 + x_2), x_1(x_3 + x_4), x_2(x_3 + x_4)\rangle \\
& = \langle x_1 + x_2, x_3 + x_4\rangle \cap \langle x_1, x_2\rangle \cap \langle x_3, x_4\rangle.
\end{align*}
 Since $(1,-1,1,-1)\in (\C^*)^4$ belongs to the variety of $\tinit_w(I)$, this ideal cannot contain a 
 monomial, meaning that the vector $w$ belongs to the tropical variety of $\mathcal{L}^{-1}$. 
\end{example}

\subsection{Positively hyperbolic toric varieties}
\label{sec:toric}

In this section we characterize tropicalizations of positively hyperbolic toric varieties given by monomial parameterizations.

 \begin{proposition}
   \label{prop:oneBinomial}
If $f = a x^\alpha + b x^\beta \in \C[x_1, \hdots, x_n]$ is stable with $a,b\neq 0$ and 
$\alpha, \beta \in (\Z_{\geq 0})^n$, then 
$\beta - \alpha \in \{\pm e_i \pm e_j\}_{i,j\in[n]} \cup \{\pm e_i\}_{i \in[n]}$. Furthermore, if $a, b$ are real and positive then $\beta - \alpha \in \{ e_i - e_j\}_{i,j\in[n]} \cup \{\pm e_i\}_{i \in[n]}$.
 Conversely, if $\beta - \alpha \in \{\pm e_i \pm e_j\}_{i,j\in[n]} \cup \{\pm e_i\}_{i \in[n]}$, then there is a choice of $b \in \{\pm1\}$ such that $x^{\alpha} + bx^{\beta} $ is stable.
 \end{proposition}

\begin{proof} We rely heavily on the stability preservers given in Lemma~\ref{lem:StablePreservers}. 
A polynomial is stable if and only if all its factors are stable, so we can take 
$x^{\alpha}$ and $x^{\beta}$ to have no common factors (meaning at most one of 
$\alpha_i, \beta_i$ is nonzero for each $i$). 
For each $i$ with $\alpha_i>0$, we can replace $x_i$ by $-1/x_i$ and 
clear denominators.  The resulting polynomial, $a + (-1)^{|\alpha|}b x^{\beta+\alpha}$, 
is stable if and only if $f$ is stable, where $|\alpha|$ denotes the
$\ell_1$-norm of $\alpha$. 
If $f$ is stable then so is $a + (-1)^{|\alpha|}b t^{|\beta+\alpha|} \in \C[t]$, which is obtained 
from $f$ by specializing $x_i=t$ for all $i$. 
The roots of this specialization are the $|\beta+\alpha|$-th roots of 
$(-1)^{|\alpha|+1}a/b$ and all lie in the complex half-plane $\{z\in \C : \im(z)\leq 0\}$, 
implying that $|\beta +\alpha| \in \{1,2\}$. 
Furthermore, if $|\beta +\alpha|=2$ then $(-1)^{|\alpha|+1}a/b$ must be a positive real number. 
This gives three options: 
\begin{itemize}
\item $\{\alpha, \beta\} = \{ 0, e_i\}$ for some $i\in [n]$,  
\item $\{\alpha, \beta\} = \{ e_i, e_j \}$ for some $i,j \in [n]$ and $a/b\in \R_{+}$, or  
\item $\{\alpha, \beta\} = \{ 0, e_i +e_j \}$ for some $i,j \in [n]$ and $a/b\in \R_{-}$. 
\end{itemize}
The stable polynomials $1-x_i$, $x_i+x_j$,  and $1-x_ix_j$
show the existence of each case.\end{proof}

Let $A\in \Z^{d\times n}$ be a rank-$d$ matrix with columns $a_1, \hdots, a_n \in \Z^d$ and $\lambda = (\lambda_1, \dots, \lambda_n) \in (\C^*)^n$.  The corresponding 
\emph{toric variety} $X_{\lambda,A}\subset \C^n$ is parametrized by the map 
  \[(\C^*)^d \rightarrow (\C^*)^n, \quad
x \mapsto (\lambda_1x^{a_1}, \dots, \lambda_n x^{a_n}).\]

\begin{definition}\label{def:noncrossing}
We say that a family of subsets $S_1,S_2,\dots, S_d \subset [n]$ is a \emph{non-crossing partition} of a subset of $[n]$ if the subsets are disjoint and there are no $a, b, c, d$ in cyclic order such that $a, c \in S_i$ and $b, d \in S_j$ for some $i \neq j$. Equivalently, if we place the numbers $1, 2, \dots , n$ on $n$ vertices around a circle in clockwise order, and for each $S_i$ we draw a polygon on the corresponding vertices, then $S_1, \dots, S_t$ is a
non-crossing partition if and only if no two of these polygons intersect; see the example in Figure \ref{fig:noncrossing}. 
\end{definition}

\begin{theorem}
\label{thm:toricStable}
Let $A$ be a $d \times n$ integer matrix of rank $d$.  There exists $\lambda \in (\C^*)^n$ such that the toric variety $X_{\lambda,A}\subset \C^n$
 is positively hyperbolic if and only if, up to the $\GL_d(\Q)$ action on the left, the columns of $A$ belong to $\{0\} \cup \{\pm e_i\}_{i \in [n]}$ 
and the supports of the rows of $A$ form a non-crossing partition of a subset of $[n]$.
In fact, one can always take $\lambda \in \{\pm 1\}^n$.
\end{theorem}

\begin{proof}
Suppose there exists $\lambda = (\lambda_1, \dots, \lambda_n) \in (\C^*)^n$ such that 
the $X_{\lambda,A}\subset \C^n$ is positively hyperbolic.  We will show that every vector of minimal support in the kernel of $A$ belongs to $\{e_i \pm e_j\}_{i, j\in [n]} \cup \{e_i\}_{i \in[n]}$ up to scaling.
If $n = d-1$,  the variety $X = X_{\lambda,A}$ is a hypersurface, and the result follows from Proposition~\ref{prop:oneBinomial}.
This shows the claim for $n=1,2$, and we proceed by induction on $n$. 

Let $v$ be a vector of minimal support in the kernel of $A$. 
If ${\rm supp}(v) = [n]$, $X$ is a hypersurface and the result follows as before. 
So we may suppose $v_{k} = 0$ for some $k \in [n]$.  

We can reduce to the case $k=n$ as follows.  
If $k< n$, consider the variety $Y = {\rm cyc}_c^{n-k}(X)$. 
Then $Y$ is the toric variety $X_{\tilde{\lambda}, \tilde{A}}$, where 
$\tilde{\lambda} = {\rm cyc}_c^{n-k}(\lambda)$ and  
$\tilde{A} = {\rm cyc}_1^{n-k}(A)$. By Corollary~\ref{cor:simplePreservers}, $Y$ is also positively hyperbolic. 
Furthermore $\tilde{v} = {\rm cyc}_1^{n-k}(v)$ is a vector of minimal support in the kernel of $\tilde{A}$ 
with $\tilde{v}_n = 0$, and $\tilde{v}$ belongs to $\{e_i \pm e_j\}_{i, j\in [n]} \cup \{e_i\}_{i \in[n]}$ up to scaling if and only if $v$ does. 

If $k = n$, then consider the projection $\pi_S(X)$ with $S = [n-1]$. 
This is the toric variety $X_{\lambda_S, A_S}$ of the submatrix $A_S$ 
with columns $a_i$ for $i\in S$, where $\lambda_S$ is the projection of $\lambda$ onto the coordinates indexed by $S$.  
Note that $v$ is a vector of minimal support in the kernel of $A_S$, and by Lemma~\ref{lem:stableProj}, 
$X_{\lambda_S, A_S}$ is positively hyperbolic. By induction, it follows that $v$ has the desired form.

We will now prove the non-crossing condition.
After row-reducing the matrix $A$, we can assume that any column $a_i$
of $A$ is equal to $0$ or $\pm e_j$ for some $j \in [d]$. The matroid
on the columns of $A$ is equal to the algebraic matroid of the toric
variety $X_{\lambda,A}$, which is a positroid by Corollary~\ref{cor:algmatroid}.
The  connected components of this positroid are the supports of the rows of $A$, so by~\cite[Theorem~7.6]{ARW1}, they form a non-crossing partition.

For the converse, because of the non-crossing condition, we can cycle the columns and permute the rows of $A$ to get a block matrix of the form   
$A = \begin{pmatrix} A_1 & 0 \\ 0 & A_2 \end{pmatrix}$, so that $X_{\lambda, A} = X_{\lambda_1, A_1} \times X_{\lambda_2, A_2}$. 
Using Lemma \ref{lem:product} and induction, we can reduce to the case when $A$ consists of a single row of $\pm 1$s.  Given that $\sgn(\im(z)) = -\sgn(\im(z^{-1}))$, we can then choose $\lambda \in \{\pm 1\}^n$ so that $\varbar(x) = n-1$ for all $x \in X_{\lambda,A}$. For instance, for $A = \begin{pmatrix}1 & -1& -1& 1& 1\end{pmatrix}$,  the parametrization $t \mapsto (t, t^{-1}, -t^{-1}, -t, t)$ gives a positively hyperbolic toric variety.
\end{proof}

\subsection{Tropicalization  of positively hyperbolic varieties}

We now give our main result about the local structure of tropicalizations of positively hyperbolic varieties.

\begin{theorem}
  \label{thm:tropical}
\label{thm:non-crossing}
  Let $X \subset \C^n$ be a positively hyperbolic variety.  Then the linear subspace parallel to any maximal face of $\trop(X)$ is spanned by $0/\pm 1$ vectors whose supports form a non-crossing partition of a subset of $[n]$. If, in addition, $X$ is homogeneous, then the linear subspace parallel to any maximal face is spanned by  $0/1$ vectors whose supports form a non-crossing partition of $[n]$.
\end{theorem}

\begin{proof}
Suppose $X = V(I)$ is a positively hyperbolic variety of dimension $d$
in $K^n$.  Let $w$ be a point in the interior of a maximal face $C$ of
$\trop(X)$.  By Proposition~\ref{prop:initial}, $V(\tinit_w(I))$ is
also positively hyperbolic. The tropicalization $\trop(V(\tinit_w(I)))$
is the linear subspace $L$ parallel to $C$, which is a linear subspace of
dimension $d$.  We claim that the corresponding $d$-dimensional
subtorus of $(\C^*)^n$ acts on $V(\tinit_w(I)) \cap (\C^*)^n$.
Suppose not.  Consider an irreducible component $X'$ of 
$V(\tinit_w(I)) \cap (\C^*)^n$, which must have dimension $d$.
 If $(\C^*)^n$ does not act on $X'$, then the orbit of $X'$ under this
 subtorus action is irreducible and strictly larger than $X'$, thus it has dimension $> d$.  On the other
 hand, the tropicalization of the orbit is equal to $\trop(X') + L$, and we have $\trop(X') + L \subset \trop(V(\tinit_w(I))) + L = L$, which  has dimension $d$, contradicting the fact that tropicalization preserves dimensions. Moreover, in $(\C^*)^n$ each orbit of the $d$-dimensional torus has dimension $d$.  It follows that $V(\tinit_w(I)) \cap (\C^*)^n$ is the torus orbit closure of finitely many points in $(\C^*)^n$.  Thus each irreducible component is generated by binomials, and the first statement follows from Theorem~\ref{thm:toricStable}.

If $X$ is homogeneous, each face of $\trop(X)$ contains $(1,1,\dots,1)$ in its lineality space.  Given that the linear subspace parallel to each face is spanned by vectors with disjoint support, it follows that those vectors must be $0/1$ vectors up to scaling, and their supports must form a non-crossing partition of $[n]$.
\end{proof}
\begin{corollary}
\label{cor:constCoeff}
If $X\subset K^n$ is a positively hyperbolic variety defined over $\C$, then 
its tropicalization is a subfan of the (type $B$) hyperplane arrangement defined
by $x_i = 0$ and $x_i = \pm x_j$.  If, in addition, $X$ is homogeneous, then $\trop(X)$ is a subfan of the (type $A$/ braid) arrangement defined by $x_i=x_j$. 
\end{corollary}

\begin{example}
Consider the cubic projective curve $X = \mathcal{L}^{-1}$ defined in Example~\ref{ex:reciprocalPlane}. 
The variety of the initial ideal of $X$ with respect to the vector $w = (\frac{1}{2},\frac{1}{2},0,-1)$ 
has three irreducible components,
namely, 
$V(\tinit_w(I)) = V(x_1 + x_2, x_3 + x_4) \cup V( x_1, x_2)\cup V(x_3, x_4)$. 
See also Example \ref{ex:reciprocalcircuits}.
The component that intersects the torus $(\C^*)^4$ is a toric variety, obtained as the (closure of the) image of the map $K^2 \to K^4$ 
given by $(a,b)\mapsto (a,-a,b,-b)$.  
Its tropicalization is the linear subspace parallel to the maximal cone of $\trop(X)$ containing $w$. 
This is the linear subspace spanned by the vectors $(1,1,0,0)$ and $(0,0,1,1)$, which 
are $0/1$ vectors supported on the non-crossing partition $\{1,2\}\sqcup \{3,4\}$. 

For $u = (1,1,0,-2)$, we find that 
$V(\tinit_u(I)) =  V(x_1 - x_3, x_2 + x_3) \cup V(x_1 + x_2, x_4)\cup V(x_3,x_4)$.
The linear span of the maximal cone of $\trop(X)$ containing $u$ is the tropicalization of the 
unique component not contained in the coordinate hyperplanes. 
This is the image of the map $K^2\to K^4$ given by $(a,b)\mapsto (a,-a,a,b)$. 
Its tropicalization is the linear subspace spanned by $(1,1,1,0)$ and $(0,0,0,1)$, which is supported on the 
non-crossing partition $\{1,2,3\}\sqcup \{4\}$. 
\end{example}

When the variety $X$ is a hypersurface defined by a polynomial $f =
0$, its tropicalization $\trop(X)$ is a codimension-one polyhedral
complex whose maximal faces are normal to the edges in the regular
subdivision of the Newton polytope of $f$ induced by the valuations of
the coefficients of $f$.  
In this case, the non-crossing condition of Theorem~\ref{thm:tropical}
is always satisfied, as the linear subspace parallel to each maximal face is
spanned by vectors with singleton supports, except possibly a single
$e_i \pm e_j$.  We now have the following result of
Choe, Oxley, Sokal, and Wagner~\cite{COSW}.

\begin{corollary}
The Newton polytope of a homogeneous multiaffine stable polynomial $f$ is a matroid polytope.
\end{corollary}

\begin{proof}
By Proposition \ref{prop:connectiontostable}, the hypersurface $V(f)$ is positively hyperbolic.
The Newton polytope of $f$ is a $0/1$ polytope because $f$ is multiaffine, and thus its edges are in directions $e_i - e_j$ by Corollary~\ref{cor:constCoeff}, which means it is a matroid polytope (see \cite[Theorem~4.1]{GGMS}).
\end{proof}

A polytope is called an {\em $M$-convex polyhedron} or a {\em generalized
permutohedron} if its edges are in directions $e_i - e_j$.
They are defined by inequalities of the form $\sum_{i \in I}x_i \geq
c_I$ for $I \subset [n]$ and $x_1+\cdots+x_n = c_{[n]}$ where the
$c_I$'s form a submodular set function with $c_\varnothing =
0$. If the polytope is integral, all the numbers $c_I$ can be chosen to
be integers. 

 A set $S \subset \Z^n$ is called {\em $M$-convex}
 if $S = P \cap \Z^n$ where $P$ is an integral $M$-convex
polyhedron~\cite[Theorem~4.15]{Murota}.   For a set $S \subset \Z^n$,
a function $f : S \rightarrow \R$ is {\em M-convex} if for every $w
\in (\R^n)^*$, the set
\[
 \{ x \in S \mid
  f(x) + \langle w,x \rangle \leq  f(y) + \langle w,y \rangle  \text{
    for all }  y \in S\}
\]
is an $M$-convex set~\cite[Theorem 6.43]{Murota}. 
We will now give a new proof the following result of  Br\"and\'en~\cite{Branden2}.  

\begin{theorem}
If $f$ is a homogeneous stable polynomial over a valued field, then
the valuation (tropicalization) of the coefficients of $f$
form an {\em $M$-convex function} on the support of~$f$.
Equivalently, for every $w \in \R^n$, the support of $\tinit_w(f)$ is an $M$-convex
set. 
\end{theorem}

\begin{proof}
 By Proposition~\ref{prop:initial}, it suffices to show that
 $\supp(g)$ is an $M$-convex set for any
 homogeneous stable polynomial $g$ with complex coefficients.  Let $P = \conv(\supp(g))$ be the Newton
 polytope of $g$.
Note that, by the homogeneous case of Corollary~\ref{cor:constCoeff}, $P$ is an $M$-convex polyhedron.

Let us give now a more self-contained proof of the previous statement.
Let $v \in \R^n$ be such that the Newton polytope of $\init_v(g)$  is
an edge of $P$.  Since $\init_v(g)$ is stable, by taking
partial derivatives we can transform $\init_v(g)$ into a stable binomial
$ax^{\alpha}+bx^{\beta}$ where the two terms involve disjoint sets of variables, and its Newton
polytope is parallel to that of $\init_v(g)$.  
As $g$ is homogeneous, so is $\init_v(g)$ and the binomial $ax^{\alpha}+bx^{\beta}$. 
Then by Proposition~\ref{prop:oneBinomial}, we must have $\alpha - \beta = e_i - e_j$ for 
some $i,j\in [n]$. 
This shows that the difference between any two consecutive terms in the support of
$\init_v(g)$ is $e_i - e_j$ for some $i \neq j$, and that
the support of a homogeneous stable polynomial is hole-free when the Newton
polytope is one-dimensional. 

It remains to show that $\supp(g) = P\cap \Z^n$.  We proceed by
induction on the dimension of~$P$.  We have already done the one-dimensional
case above. Let $u \in P \cap \Z^n$.  If $u$ is on the boundary of
$P$, then we can reduce dimension by taking initial
forms and conclude that $u \in \supp(g)$. 
We may thus assume that all points on the boundary of $P$ are in
$\supp(g)$. For any integer point $u$ in the interior of $P$,  we claim that $u + e_i -
e_j \in P$ for all $i \neq j$.  Indeed, any inequality of the form $\sum_{i
  \in I}x_i \leq c_I$ which is valid on $P$ is satisfied by
the interior point $u$ strictly, thus $\sum_{i \in I}u_i \leq c_I-1$, so $u + e_i -
e_j$ satisfies $\sum_{i
  \in I}x_i \leq c_I$.  It follows that, for every $i\neq j$, 
$u$ is in the interior of a segment in direction $e_i - e_j$ whose end
points are integer points on the boundary of $P$, which are in $\supp(g)$ as seen above.  Now
consider $\partial_1^{u_1}(g)$ where $\partial_1$ denotes $\frac{\partial}{\partial x_1}$.  
The point $u-u_1e_1$ is in the Newton polytope of $\partial_1^{u_1}(g)$, because
it is still the interior of a segment in
direction $e_2-e_3$ with endpoints in the support of
$\partial_1^{u_1}(g)$. As $u-u_1
e_1$ is on the boundary of the Newton
polytope of $\partial_1^{u_1}(g)$, we conclude that it is in the support of $\partial_1^{u_1}(g)$, and so $u \in \supp(g)$ as desired.
\end{proof}

\begin{remark}
Note that the above proof extends to any class of homogeneous 
polynomials $\mathcal{S} \subseteq K[x_1, \hdots, x_n]$ that  
\begin{enumerate}
\item is closed under taking $t$-initial forms ($f\in \mathcal{S} \Rightarrow \tinit_w(f)\in \mathcal{S}$),
\item is closed under taking derivatives ($f\in \mathcal{S} \Rightarrow \partial f/\partial x_j \in \mathcal{S}$ for $j\in [n]$), and
\item contains binomials only of the form $ax^{\alpha}+bx^{\beta}$ where $\alpha - \beta=e_i - e_j$ for some $i,j\in [n]$. 
\end{enumerate} 
In particular, this provides another method of proving that the valuations of 
the coefficients of a \emph{Lorentzian} or \emph{strongly log-concave} polynomial 
give an $M$-convex function.  See \cite[Theorem~8.7]{BH19}. 
A polynomial $f$ is called \emph{strongly log-concave} if $f$ and all of its derivatives 
$(\frac{\partial}{\partial x_1})^{\alpha_1} \cdots (\frac{\partial}{\partial x_n})^{\alpha_n}f$ 
are identically zero or log-concave on the positive orthant.  
Since these are closed under positive scaling
$f \mapsto c f(\lambda_1x_1, \hdots, \lambda_n x_n)$ for $c \in \R_{>0}$ and $\lambda \in \R_{>0}^n$, 
and under taking limits, we can see that, over the field of locally-convergent Puiseux series, 
the set of strongly-log concave polynomials is closed under taking initial forms, since 
 \[\tinit_w(f) = \lim_{t\rightarrow 0} f(t^{w_1} x_1, \dots, t^{w_n} x_n) t^a \text{ where }a = - \min\{w \cdot s \mid s \in
\supp(f)\}.\]  
One can also check directly from taking Hessians that the only 
homogeneous binomials $a x^{\alpha} + b x^{\beta}$ that are log-concave on the nonnegative orthant 
have $|\alpha - \beta|=2$.
A proof that strongly log-concave polynomials have full support in their Newton polytopes 
also appears in \cite[Corollary 4.5]{Gurvits}. 
\end{remark}

\subsection{Chow polytopes}
For a variety $X \subset \P^{n-1}$ of dimension $d-1$, we can
associate a hypersurface in the Grassmannian $\Gr(n-d,n)$ consisting
of all $(n-d-1)$-dimensional linear subspaces in $\P^{n-1}$ that intersect $X$.
The defining polynomial of this hypersurface, in 
the coordinate ring of $\Gr(n-d,n)$,
is called the {\em Chow form} of $X$~\cite{KSZ}.  
The projection of its Newton polytope under the map $\R^{\binom{[n]}{n-d}}
\rightarrow \R^n$ given by $e_I \mapsto \sum_{i \in I}e_i$  is called
the {\em Chow polytope} of $X$.  
While the Chow form of $X$ is unique only up to Pl\"ucker relations, 
its Chow polytope is well defined. 
If $X$ is a linear subspace, then the
Chow polytope of $X$ is its matroid polytope.  It follows from the definition
that the Chow form of $X$ is the product of the Chow forms of its
irreducible components, so its Chow polytope is the Minkowski sum of
those of the components.

\begin{theorem}\label{thm:Chow}
The Chow polytope of any homogeneous positively hyperbolic variety is
a generalized permutohedron.  
\end{theorem}

\begin{proof}
Let us first state Fink's ``orthant shooting'' construction of the
Chow polytope from a tropical variety~\cite[Theorem~5.1]{Fink}. Recall
that the tropical hypersurface of a polytope is the union of the
normal cones to the edges, that is, the union of all {\em walls} between
maximal cones in the normal fan.  For an irreducible homogenenous variety $X$ of
codimension $c$ in $K^n$ which intersects the torus $(K^*)^{n}$, the
tropical hypersurface of its Chow polytope is the Minkowski sum of the
constant coefficient tropical variety $\trop(X) \subset \R^n$ and the union of the
negatives of all $(c-1)$-dimensional cones in the normal fan of the standard
simplex~$\conv\{e_1,e_2,\dots,e_n\}$.  

Now, let $X$ be a homogeneous positively hyperbolic variety in $K^n$ with
defining ideal $I$ and Chow polytope $P$.  If all irreducible
components of $X$ intersect $(K^*)^n$, then by Fink's construction, the
linear span of any wall in the normal fan of $P$ is spanned by a coordinate orthant and a maximal cone of
$\trop(I)$.  By Theorem~\ref{thm:tropical}, any maximal cone
of $\trop(I)$ is spanned by $0/1$ vectors 
whose supports form a (non-crossing) partition of $[n]$ with $n-c$
parts. Then the wall is spanned by $0/1$ vectors whose supports form a
partition of $[n]$ with $n-1$ parts, so it is defined by an equation
of the form $x_i = x_j$.  It follows that the edges of the Chow polytope are in directions
$e_i - e_j$, and so it is a generalized permutohedron.

If $X$ has a component in a coordinate subspace, then we can apply the
orthant shooting construction in that subspace, and it follows from the same
argument that the Chow polytope of that component is a generalized
permutohedron living in the smaller subspace.   The Chow polytope of
$X$ is thus a Minkowski sum of generalized permutohedra, so it is itself
a generalized permutohedron. 
\end{proof}

\begin{remark}
While working over a field with a
non-trivial valuation, there is a natural regular subdivision of the Chow
polytope to which the tropical variety is dual \cite[Definition 4.11]{Fink}. For positively
hyperbolic projective varieties, every face of this Chow polytope subdivision is a generalized
permutohedron, by the same argument as above.
\end{remark}

\begin{example}\label{ex:Chow}
Following \cite[Example 4.2]{KV}, we find that the Chow form of the 
cubic curve $X = \mathcal{L}^{-1}$ defined in Example~\ref{ex:reciprocalPlane} is given by
\[\sum_{T\in \mathcal{T}} \, \prod_{k\ell\in \mathcal{T}}q_{k\ell} \, \prod_{ij\not\in T}p_{ij} \,\, \in \R\{\!\{t\}\!\}[p_{ij}]/I(\Gr(2,4)),\]
where $\mathcal{T}$ is the set of spanning trees of the complete graph on four vertices
and $q_{k\ell} = t^{4-\ell} - t^{4-k}$ is the $(k,\ell)$th Pl\"ucker coordinate of $\mathcal{L}^{\perp}$. 
The Chow polytope is the Newton polytope of the polynomial $f$
in $\R\{\!\{t\}\!\}[x_1, \hdots, x_4]$ obtained by substituting $p_{ij}\to x_ix_j$. 
In this case, the Chow polytope is just the simplex given by the convex hull 
of $(2,2,2,2)-2e_i$ for $i=1, \hdots, 4$.  The valuations of the coefficients of $f$ induce a regular subdivision of the 
Chow polytope whose faces are generalized permutohedra, and the (negatives of the) faces of the tropical variety $\trop(X)$ are dual
to some of the faces of this subdivision. See Figure~\ref{fig:Chow}.
\begin{figure}[ht]
\begin{center}
\includegraphics[height=1.4in]{PIX/RecipLine.pdf} \quad \quad \includegraphics[height=1.4in]{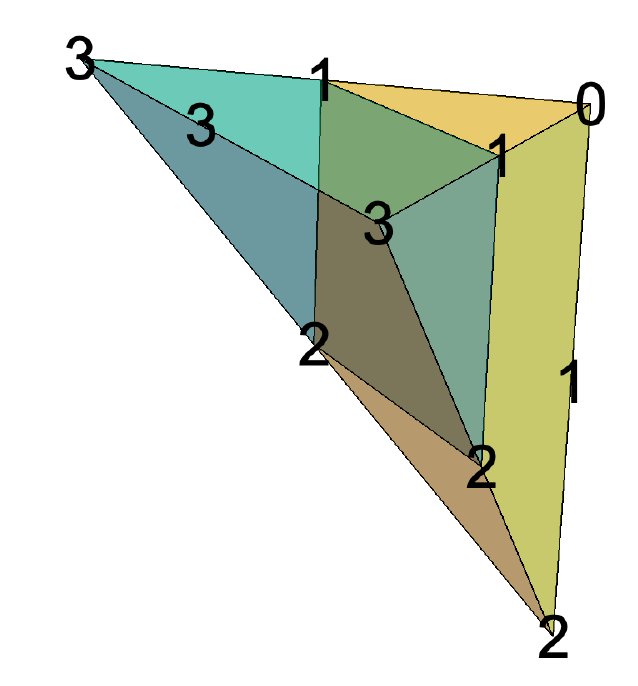} 
\end{center}
\caption{The tropical variety and subdivided Chow polytope of Example~\ref{ex:Chow}.} \label{fig:Chow}
\end{figure}

The cone $C = (1,1,0,-1) + \R_{\geq0}(0,0,0,-1) + \R(1,1,1,1)$ of $\trop(X)$ in direction $-e_4$ corresponds to the 
initial ideal $\langle x_1 - x_3, x_2 + x_3 \rangle \cap \langle x_1 + x_2, x_4 \rangle \cap \langle x_3,x_4 \rangle$.
Its Chow polytope is the Minkowski sum of the Chow polytopes of each of these 
irreducible components, namely
\[
{\rm conv}\{e_1+e_2,e_1+e_3,e_2+e_3\}  + {\rm conv}\{e_1+e_4,e_2+e_4\} +  \{e_3+e_4\}. 
\]
Since each of these simplices is a generalized permutohedron, so is their Minkowski sum.
The vertices of this trapezoid correspond to coefficients with valuations $2$ or $3$. 
Each of the edges of this trapezoid is dual to a cone of the form $C + \R_{\leq 0}\cdot e_i$ for $i=1,2,3$, obtained by ``orthant shooting''. 
The edge with valuation $3$,  ${\rm conv}\{(2,0,2,2),(0,2,2,2)\}$, is dual to $C + \R_{\leq 0}\cdot e_3$, 
and so is the parallel edge ${\rm conv}\{(2,1,1,2),(1, 2, 1,2)\}$ with constant valuation $2$. 
\end{example}

\section{Curves}\label{sec:curves}

In this section we focus on projective curves, and prove that in this case, a certain converse of Theorem \ref{thm:tropical} holds.  A (constant-coefficient, homogeneous) tropical curve is a balanced one-dimensional fan in
$\R^n/(1,1,\dots,1)$, or equivalently, a two-dimensional fan in $\R^n$
with lineality space containing $(1,1,\dots,1)$. 

For a positively hyperbolic curve $X \subset \P^{n-1}$, its
tropicalization $\trop(X)$ is a tropical curve where
 each ray is spanned by
a $0/1$ vector in which all $1$'s are cyclically
consecutive, by Theorem~\ref{thm:tropical}.
We will prove that all such fans can be obtained as tropicalizations of positively hyperbolic curves. 

We first demonstrate in a concrete example how to construct a positively hyperbolic curve with such a given tropicalization.

\begin{example}
  \label{ex:curve}
Consider the tropical curve $C$ in $\R^6/(1,1,\dots,1)$ whose rays are the rows the following matrix, all with multiplicity $1$:
\[
\bordermatrix{& 1 & 2&3&4&5&6 \cr
  1 & 0&1&1&1&0&0 \cr
  2 & 1&1&0&0&1&1 \cr
  3 & 0&0&1&1&1&0 \cr
  4 & 1&0&0&0&0&1 \cr
}
\]
The rows of the matrix are ordered in such a way that whenever a block of $1$'s ends in a row, the block of $1$'s in the row below begins.  
By Corollary~\ref{cor:simplePreservers}, cyclically permuting variables does not change whether or not $C$ is the tropicalization of a positively hyperbolic curve in $\P^5$. We can thus cyclically permute the columns of the matrix to get the following matrix instead, where the first row contains $1$'s followed by $0$'s.
\[
\bordermatrix{& 1 & 2&3&4&5&6 \cr
  1 & 1&1&1&0&0&0 \cr
  2 & 1&0&0&1&1&1 \cr
  3 & 0&1&1&1&0&0 \cr
  4 & 0&0&0&0&1&1 \cr
}
\]

This new tropical curve is the tropicalization of the projective curve parametrized by $(u,v) \mapsto $
\[
((u-v)(u-2v), (u-v)(u-3v), (u-v)(u-3v), (u-2v)(u-3v), (u-2v)(u-4v), (u-2v)(u-4v))
\]
as in Speyer's construction~\cite{Speyer}.
Note that each coordinate corresponds to a column of the matrix above. The balancing condition says that the column sums are the same, so this parametrization is homogeneous.
The coefficients of $u$ and $v$ can be chosen arbitrarily, as long as they are consistent, to obtain the desired tropical curve. 

We will see that the following modified parametrization, with alternating signs, which gives the same tropical curve, is a positively hyperbolic curve: 
\[
(u,v) \mapsto (s_1s_2, -s_1s_3, s_1s_3, -s_2s_3, s_2s_4, -s_2s_4),
\]
where $s_k = u - k\, v$ for $k=1,2,3,4$. 
For this, we will show that for any point in the image we have
\[\var(\im(s_1s_2), \im(s_1s_3), \im(s_1s_3), \im(s_2s_3), \im(s_2s_4), \im(s_2s_4)) \leq 1.\] 
This implies that 
\[\varbar(\im(s_1s_2), \im(- s_1s_3), \im(s_1s_3), \im(- s_2s_3), \im(s_2s_4),  \im(- s_2s_4)) \geq 4,\]
as required.

The sequence $s_1s_2, s_1s_3,s_1s_3, s_2s_3, s_2s_4,s_2s_4$ has the following properties:  
\begin{enumerate}
\item At each step along the sequence, either the expression stays the same, or one (and only one) $s_i$ is replaced by $s_{i+1}$. (We do not replace $s_4$ by $s_1$).
  \item Each replacement $s_i \rightarrow s_{i+1}$ occurs exactly once.
  \end{enumerate}

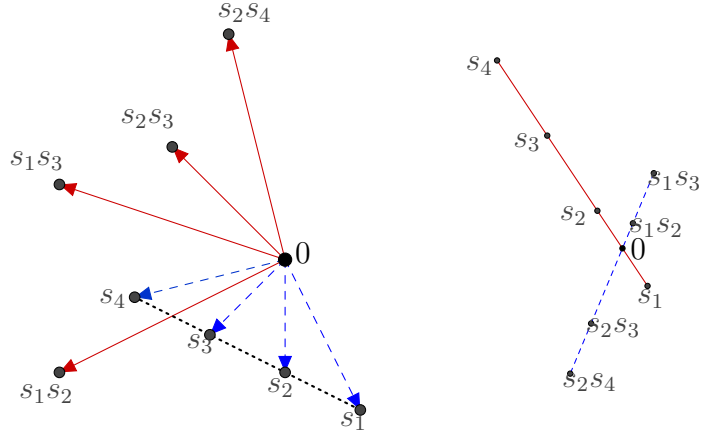
\begin{figure}

  \hspace{-3cm}
  \begin{tabular}{lr}
 
    \begin{minipage}{0.5\linewidth}
      \begin{center}
              \definecolor{qqttcc}{rgb}{0.,0.2,0.8}
\definecolor{qqqqff}{rgb}{0.,0.,1.}
\definecolor{ccqqqq}{rgb}{0.8,0.,0.}
\definecolor{uuuuuu}{rgb}{0.26666666666666666,0.26666666666666666,0.26666666666666666}

\begin{tikzpicture}[line cap=round,line join=round,>=triangle 45,x=1.0cm,y=1.0cm]
\begin{axis}[
x=1.0cm,y=1.0cm,
axis lines=middle,
xmin=-8.998979293420078,
xmax=6.377442747842604,
ymin=-3.028105818563085,
ymax=4.085833342295479,
axis line style={draw=none},
tick style={draw=none},
ticks=none
  ] 
\clip(-8.998979293420078,-3.028105818563085) rectangle (6.377442747842604,4.085833342295479);
\draw [->,line width=0.4pt,color=ccqqqq] (0.,0.) -- (-3.,-1.5);
\draw [->,line width=0.4pt,color=ccqqqq] (0.,0.) -- (-3.,1.);
\draw [->,line width=0.4pt,color=ccqqqq] (0.,0.) -- (-1.5,1.5);
\draw [->,line width=0.4pt,color=ccqqqq] (0.,0.) -- (-0.75,3.);
\draw [->,line width=0.4pt,dash pattern=on 3pt off 3pt,color=qqqqff] (0.,0.) -- (1.,-2.);
\draw [->,line width=0.4pt,dash pattern=on 3pt off 3pt,color=qqqqff] (0.,0.) -- (0.,-1.5);
\draw [->,line width=0.4pt,dash pattern=on 3pt off 3pt,color=qqqqff] (0.,0.) -- (-1.,-1.);
\draw [->,line width=0.4pt,dash pattern=on 3pt off 3pt,color=qqttcc] (0.,0.) -- (-2.,-0.5);
\draw [line width=0.8pt,dotted] (1.,-2.)-- (-2.,-0.5);
\begin{scriptsize}
\draw [fill=uuuuuu] (1.,-2.) circle (2.0pt);
\draw[color=uuuuuu] (0.9101402461283957,-2.149118278094564) node {$s_1$};
\draw [fill=uuuuuu] (0.,-1.5) circle (2.0pt);
\draw[color=uuuuuu] (-0.06260596532343549,-1.7037645909238464) node {$s_2$};
\draw [fill=uuuuuu] (-1.,-1.) circle (2.0pt);
\draw[color=uuuuuu] (-1.1291108477585756,-1.1177728972781655) node {$s_3$};
\draw [fill=uuuuuu] (-2.,-0.5) circle (2.0pt);
\draw[color=uuuuuu] (-2.2893744011770254,-0.5200613697595714) node {$s_4$};
\draw [fill=uuuuuu] (-3.,-1.5) circle (2.0pt);
\draw[color=uuuuuu] (-3.174221858582004,-1.7799435110977848) node {$s_1 s_2$};
\draw [fill=uuuuuu] (-0.75,3.) circle (2.0pt);
\draw[color=uuuuuu] (-0.5372592371764375,3.283024722000896) node {$s_2 s_4$};
\draw [fill=black] (0.,0.) circle (2.5pt);
\draw[color=black] (0.23624979843586208,0.09522990856839329) node {$0$};
\draw [fill=uuuuuu] (-3.,1.) circle (2.0pt);
\draw[color=uuuuuu] (-3.303140031184054,1.3140926313514092) node {$s_1 s_3$};
\draw [fill=uuuuuu] (-1.5,1.5) circle (2.0pt);
\draw[color=uuuuuu] (-1.8264409631969367,1.8766446572512625) node {$s_2 s_3$};
\end{scriptsize}
\end{axis}
\end{tikzpicture}
      \end{center}
    \end{minipage}
& 
    \begin{minipage}{0.5\linewidth}
      \begin{center}    
        \definecolor{qqqqff}{rgb}{0.,0.,1.}
\definecolor{ccqqqq}{rgb}{0.8,0.,0.}
\definecolor{uuuuuu}{rgb}{0.26666666666666666,0.26666666666666666,0.26666666666666666}

\begin{tikzpicture}[line cap=round,line join=round,>=triangle 45,x=1.0cm,y=1.0cm]
\begin{axis}[
x=1.0cm,y=1.0cm,
axis lines=middle,
xmin=-4.877596126956554,
xmax=6.980270615616448,
ymin=-2.3643576755934466,
ymax=3.12171329448417,
xtick={-4.5,-4.0,...,6.5},
ytick={-2.0,-1.5,...,3.0},
  axis line style={draw=none},
tick style={draw=none},
ticks=none
]
\clip(-4.877596126956554,-2.3643576755934466) rectangle (6.980270615616448,3.12171329448417);
\draw [line width=0.4pt,color=ccqqqq] (0.33333333333333337,-0.5)-- (-1.6666666666666665,2.5);
\draw [line width=0.4pt,dash pattern=on 2pt off 2pt,color=qqqqff] (0.41666666666666663,1.)-- (-0.6944444444444446,-1.6666666666666665);
\begin{scriptsize}
\draw [fill=uuuuuu] (0.33333333333333337,-0.5) circle (1.0pt);
\draw[color=uuuuuu] (0.3780056343286598,-0.6652121362613346) node {$s_1$};
\draw [fill=uuuuuu] (-0.33333333333333326,0.5) circle (1.0pt);
\draw[color=uuuuuu] (-0.6523272990961285,0.4374248626669508) node {$s_2$};
\draw [fill=uuuuuu] (-1.,1.5) circle (1.0pt);
\draw[color=uuuuuu] (-1.2669118558758268,1.4135297469641215) node {$s_3$};
\draw [fill=uuuuuu] (-1.6666666666666665,2.5) circle (1.0pt);
\draw[color=uuuuuu] (-1.8995724290313984,2.425786664013039) node {$s_4$};
\draw [fill=uuuuuu] (0.1388888888888889,0.3333333333333333) circle (1.0pt);
\draw[color=uuuuuu] (0.44579069573818536,0.2702217111901206) node {$s_1 s_2$};
\draw [fill=uuuuuu] (-0.6944444444444446,-1.6666666666666665) circle (1.0pt);
\draw[color=uuuuuu] (-0.43993410667961513,-1.781406147471525) node {$s_2 s_4$};
\draw [fill=black] (0.,0.) circle (1.0pt);
\draw[color=black] (0.21080248285183012,-0.009956542635919094) node {$0$};
\draw [fill=uuuuuu] (0.41666666666666663,1.) circle (1.0pt);
\draw[color=uuuuuu] (0.6898169168124773,0.9028822843456943) node {$s_1 s_3$};
\draw [fill=uuuuuu] (-0.41666666666666674,-1.) circle (1.0pt);
\draw[color=uuuuuu] (-0.12360382010182923,-1.0583654924365837) node {$s_2 s_3$};
\end{scriptsize}
\end{axis}
\end{tikzpicture}
        \end{center}
    \end{minipage}
\end{tabular}

  \caption{In Example~\ref{ex:curve} we consider the curve parameterized by $(u,v) \mapsto (s_1s_2, s_1s_3,s_1s_3, s_2s_3, s_2s_4,s_2s_4)$ where $s_k = u - kv$.  For any choice of $u$ and $v$, the imaginary part of $(s_1s_2, s_1s_3,s_1s_3, s_2s_3, s_2s_4,s_2s_4)$ has sign variation at most one.  The figure shows the points on a complex plane, in the two cases based on whether $u$ and $v$ have different or the same arguments.}
  
   \label{fig:angles}
   \end{figure}

Note that the numbers $s_1,s_2,s_3,s_4$ lie along a line, in that order, in the complex plane. We consider two cases, as depicted in Figure~\ref{fig:angles}.
First, suppose the line through $s_1,s_2,s_3,s_4$ does not contain $0$, so the numbers $s_1s_2, s_1s_3, s_1s_3, s_2s_3, s_2s_4, s_2s_4$ are all nonzero.  By the two properties stated above, the arguments of the complex numbers $s_1s_2, s_1s_3, s_1s_3, s_2s_3, s_2s_4, s_2s_4$ only move in one direction, and the total change is strictly less than $\pi$.  Thus their imaginary parts contain at most one $0$ and change sign at most once.

Next, assume that the line through $s_1,s_2,s_3,s_4$ contains $0$, i.e.\ $u$ and $v$ have the same argument.  Then the numbers $s_1s_2, s_1s_3, s_1s_3, s_2s_3, s_2s_4, s_2s_4$ all lie on the same line through $0$.  Since each replacement $s_i \rightarrow s_{i+1}$ is used only once and the only replacements that induce a sign change are when $s_i$ and $s_{i+1}$ are on opposite sides of $0$ (or one of them is $0$), it follows that the $0$'s occur consecutively and that the sign variation is at most~$1$.
\end{example}

\begin{theorem}
  \label{thm:curves}
A tropical curve $C$ in $\R^n/(1,1,\dots,1)$ is the tropicalization of a positively hyperbolic curve in $\P^{n-1}$ if and only if each ray of $C$ is spanned by a 0/1 vector in which the 1's appear in a cyclically consecutive block.
\end{theorem}

\begin{proof}
The ``only if'' direction follows from Theorem~\ref{thm:tropical}.
For the ``if'' direction, since the union of two positively hyperbolic curves is positively 
hyperbolic, it suffices to prove the statement in the case where the tropical curve $C$ is {\em irreducible}, i.e., it is not the union of two other curves. 
Consider an irreducible constant-coefficient homogeneous tropical curve in which each ray is a 0/1 vector with 1's appearing in a cyclically consecutive block.  Take any ray of the tropical curve; suppose the $1$'s appear in coordinates forming a consecutive block $\{k_1,\dots,k_2-1\}$. By the balancing condition, the number of such blocks ending at $k_2-1$ is equal to the number of blocks starting at $k_2$.  So there must be another ray where the $1$'s appear in a block $\{k_2,\dots,k_3-1\}$. Similarly, there must be another ray where the $1$'s appear in a block $\{k_3,\dots,k_4-1\}$, and so on.  We continue this way until we meet a block $\{k_\ell,\dots, k_1-1\}$.  Irreducibility ensures that there is no repetition among the numbers $k_1, k_2, \cdots, k_m$.   It is natural to think of this as a cycle $(k_1~k_2~\cdots~k_m)$ in the symmetric group $S_n$ on $n$ letters. By Corollary~\ref{cor:simplePreservers}, after a cyclic permutation of coordinates, we may assume that $k_1 = 1$. 
In Example \ref{ex:curve}, for instance, the original tropical curve corresponds to the cycle $(2~5~3~6)$ with $n=6$, which after cyclic permutation becomes the cycle $(1~4~2~5)$.

Now we consider Speyer's parametrization.  Let us put the rays of the tropical curve as rows of a matrix. For each column, we get the homogeneous polynomial in two variables $u$ and $v$ equal to $\Pi_{k}(u-k\,v)$, where $k$ runs over the positions with a $1$ in that column. See Example \ref{ex:curve}. Denote $s_k := u-k\,v$. 

Note that, by construction, when reading the matrix $A$ from left to right, if the block of $1$'s in row $i$ ends at column $j$, the block of $1$'s in row $i+1$ begins in column $j+1$. Since the cycle begins with $1$, the block of $1$'s in the last row ends in the last column, and so the indices do not need to be taken modulo $n$.
This means that, at each step along the sequence in the parametrization, either the expression stays the same, or one (and only one) $s_i$ is replaced by $s_{i+1}$.  Moreover, 
each replacement $s_i \rightarrow s_{i+1}$ occurs exactly once in the sequence.

Using the same argument as in Example \ref{ex:curve}, it follows that putting alternating signs in Speyer's parametrization results in a positively hyperbolic curve whose tropicalization is the desired tropical curve.
\end{proof}

The nonconstant coefficient analogue of Theorem~\ref{thm:curves} is false, at least when we take multiplicities into account.  There exist nonconstant coefficient genus-one tropical curves satisfying the $0/1$ non-crossing condition which are not even realizable as the tropicalization of any genus-one algebraic curve, let alone a positively hyperbolic one.
For example, consider a cubic honeycomb tropical curve in the plane as in Figure~1 of~\cite{ChanSturmfels}.  Embed the plane as the $xy$-plane $\{(x,y,0): x,y\in \R\}$ in $\R^3$.  Cut each of the three rays in $(-1,-1,0)$ direction at a random point, and attach two infinite rays in directions $(-1,-1,-1)$ and $(0,0,1)$ to obtain a balanced curve.  Unless the cut locations satisfy the ``well-spacedness condition,'' which says the minimum among their distances to the cycle is attained at least twice, the new tropical curve is not realizable as the tropicalization of any genus-one curve over characteristic~$0$~\cite[Proposition~9.2]{Speyer}.  However, the new curve has edges in directions $(1,0,0),(0,1,0),(0,0,1),(1,1,0)$, and $(1,1,1)$, so it satisfies the condition of Theorem~\ref{thm:curves}.  We do not know, however, whether such a curve is set-theoretically realizable as the tropicalization of a (positively hyperbolic) curve of higher degree (i.e., if we disregard multiplicities).

\section{Bergman Fans of Positroids}\label{sec:Positroids}

In this section we study tropicalizations of linear subspaces that are positively hyperbolic, and we prove that they are characterized by the non-crossing condition of Theorem \ref{thm:tropical}.

As discussed before, any linear subspace $L \subset K^n$ has an {\em associated matroid} $M_L$ that encodes the linear dependencies among the coordinate functions in $L$. More concretely, the circuits of $M_L$ correspond to minimal subsets of coordinates that support a linear relation valid in $L$, that is,  
the set of circuits $\mathcal C(M_L)$ of $M_L$ is the collection of inclusion-minimal subsets $C \subset [n]$ for which there is a $c \in K^n$ with $C = \supp(c)$ and $c \cdot x = 0$ for all $x \in L$.

In the constant-coefficient case, the tropicalization $\trop(L) \subset \R^n$ is a polyhedral fan that remembers exactly the matroid $M_L$. Specifically, $\trop(L)$ is equal to 
\[\trop(L) = \{x \in \R^n : \min \{ x_c : c \in C \} \text{ is attained at least twice for all } C \in \mathcal C(M_L)\},\] 
with weights equal to one in all its maximal cones. 

Bergman fans are polyhedral fans that generalize tropicalizations of linear subspaces. Every matroid $M$ on the ground set $[n]$, whether it arises from a linear subspace or not, gives rise to the \emph{Bergman fan}
\[\mathcal B(M) := \{x \in \R^n : \min \{ x_c : c \in C \} \text{ is attained at least twice for all } C \in \mathcal C(M)\}.\]
Note that if $M$ has a loop then $\mathcal B(M) = \emptyset$. These polyhedral fans encode matroids in a geometric way, and have become very important objects in the study of both matroid theory and tropical geometry. One of their essential properties is that they are balanced polyhedral fans when endowed with the same weight on all their maximal cones.

Another related and very useful point of view on matroids arises from their associated polytopes. For any matroid $M$ on $[n]$, its \emph{matroid polytope} $\Gamma_M$ is the convex hull of the indicator vectors of the bases of~$M$:
\[
\Gamma_M := \convex\{e_B \mid B \text{ is a basis of }M\} \subset \R^n,
\]
where $e_B := \sum_{i \in B} e_i$, and $\{e_1, \dots, e_n\}$ is the standard basis of $\R^n$. The dimension of $\Gamma_M$ is equal to $n - m$, where $m$ is the number of connected components of $M$. The matroid polytope $\Gamma_M$ encodes all the information of $M$, and we will sometimes identify $M$ with its polytope $\Gamma_M$.

The Bergman fan of a matroid $M$ is naturally a subfan of the outer normal fan of $\Gamma_M$.
Indeed, a vector $c \in \R^n$ lies in $\mathcal B(M)$ if and only if the face of $\Gamma_M$ that maximizes the linear form $c \cdot x$ is the matroid polytope of a loopless matroid. This fan structure is called the {\em coarse subdivision} of $\mathcal B(M)$.

As seen in Section~\ref{sec:background}, a rank-$d$ matroid $M$ is called a \emph{positroid} if it is equal to $M_L$ for some linear subspace $L \in \Gr_{\geq 0}(d,n)$.  We will say that a matroid on $[n]$ is a {\em non-crossing matroid} if the ground sets of its connected components form a non-crossing partition of $[n]$. Ardila, Rincón, and Williams showed that all positroids are non-crossing matroids \cite[Theorem 7.6]{ARW1}.


Our goal in this section is to determine which Bergman fans can 
be obtained as tropicalizations of positively hyperbolic varieties, even if we allow weights that are not necessarily equal to one. 

\begin{theorem}\label{thm:mainequivalencepositroids}
Let $M$ be a rank-$d$ loopless matroid on the ground set $[n]$. The following are equivalent statements:
\begin{enumerate}[label=(\alph*)]
\item\label{it:hyperbolic} The Bergman fan $\mathcal B(M)$ (possibly taken with weight on its maximal cones greater than one) is the tropicalization of a positively hyperbolic variety.
\item\label{it:maxcones} The linear subspace parallel to any maximal cone in $\mathcal B(M)$ is spanned by 0/1 vectors whose supports form a non-crossing partition of $[n]$.
\item\label{it:positroid} $M$ is a positroid.
\end{enumerate}
\end{theorem}
Note that, in particular, the implication \ref{it:maxcones} $\implies$ \ref{it:hyperbolic} is the converse of Theorem \ref{thm:non-crossing} in the case of linear varieties.

\begin{remark}\label{rem:maximal}
Our proof of Theorem \ref{thm:mainequivalencepositroids} will in fact show a slightly stronger statement: If $M$ is a rank-$d$ loopless matroid on the ground set $[n]$ with $m$ connected components, then for any $m < k \leq d$, condition \ref{it:maxcones} can be substituted by
\begin{enumerate}
\item[(b')] The linear subspace parallel to any $k$-dimensional cone in the coarse subdivision of $\mathcal B(M)$ is spanned by 0/1 vectors whose supports form a non-crossing partition of $[n]$.
\end{enumerate}
\end{remark}


Before proving Theorem \ref{thm:mainequivalencepositroids}, we need a suitable characterization of positroids in terms of faces of their matroid polytope. 

\begin{lemma}\label{lem:looplessfacet}
Suppose $M$ is a connected matroid on the ground set $E$ of rank at least 2, and let $a,b \in E$. Then there exists a loopless facet $F$ of the matroid polytope $\Gamma_M$ such that $a$ and $b$ are in the same connected component of $F$.
\end{lemma}
\begin{proof}
Since $M$ is a connected matroid, there exist bases $B_1, B_2$ of $M$ such that their symmetric difference $B_1 \Delta B_2$ is equal to $\{a,b\}$. As $M$ has rank at least $2$, there is $c \in B_1 \cap B_2$. Let $S$ be the parallelism class of $c$, and let $F'$ be the face of $\Gamma_{M}$ obtained by intersecting $\Gamma_{M}$ with the supporting hyperplane described by the equation $\sum_{i \in S} x_i = 1$. By the greedy algorithm, the matroid associated to $F'$ is $M|S \oplus M/S$. As $M$ is a connected matroid of rank at least 2, $F'$ is a proper face of $\Gamma_{M}$. Also, since $S$ is a parallelism class, $F'$ is a loopless face of $\Gamma_M$. Moreover, $B_1$ and $B_2$ are bases of $F'$, and thus $a$ and $b$ are in the same connected component of $F'$. Letting $F$ be any facet of $\Gamma_{M}$ containing $F'$, we see that $F$ is a loopless facet in which $a$ and $b$ are in the same connected component.
\end{proof}

\begin{lemma}\label{lem:noncrossing}
Suppose $M$ is a loopless matroid of rank $d$ with $m$ connected components, and $m < d$. If all the loopless facets of $\Gamma_M$ are non-crossing matroids then $M$ is a non-crossing matroid.
\end{lemma}
\begin{proof}
Let $M = M_1 \oplus M_2 \oplus \dots \oplus M_m$ be the decomposition of $M$ into connected components, and denote by $E_i$ the ground set of the matroid $M_i$. Suppose for a contradiction that $E_1, E_2, \dots, E_m$ is a crossing partition of $[n]$. Without loss of generality, we can assume that the sets $E_1$ and $E_2$ cross. Since $M$ is a loopless matroid with fewer connected components than its rank, one of the $M_i$ must have rank at least 2. Suppose first that one of $M_3, M_4, \dots, M_m$ has rank at least 2, say $M_3$. By Lemma \ref{lem:looplessfacet}, there is a loopless facet $F$ of $\Gamma_{M_3}$. In this case, the loopless facet $\Gamma_{M_1} \times \Gamma_{M_2} \times F \times \Gamma_{M_4} \times \dots \times \Gamma_{M_m}$ of $\Gamma_M$ is not a non-crossing matroid, as $E_1$ and $E_2$ are still parts of the corresponding partition of $[n]$. Assume now that one of $M_1$ and $M_2$ has rank 2 or more, say $M_1$. As $E_1$ and $E_2$ cross, there exist $a,b \in E_1$ and $c, d \in E_2$ such that $\{a,b\}$ and $\{c,d\}$ cross. By Lemma \ref{lem:looplessfacet}, there is a loopless facet $F$ of $\Gamma_{M_1}$ in which $a$ and $b$ are in the same connected component. But then the loopless facet $F \times \Gamma_{M_2} \times \dots \times \Gamma_{M_m}$ of $\Gamma_M$ is not a non-crossing matroid, as $\{a,b\}$ and $\{c,d\}$ are still crossing subsets of two different connected components.
\end{proof}
We note that the assumption that $m < d$ is necessary in the statement of Lemma \ref{lem:noncrossing}, as for instance there exist matroids with $m \geq d$ which are not positroids but do not have any loopless facet. Once such matroid is the rank-$2$ matroid on the set $[4]$ with bases $\{12,14,23,34\}$, which has $2$ connected components (on the sets $\{1,3\}$ and $\{2,4\}$).

The class of positroid polytopes is closed under taking faces \cite[Corollary 5.7]{ARW1}. Since any positroid is a non-crossing matroid, it follows that any face of a positroid polytope must be a non-crossing matroid. Conversely, it was shown in \cite[Proposition 5.6]{ARW1} that a matroid $M$ is a positroid if all the facets of $\Gamma_M$ are non-crossing matroids. The following proposition is a strengthening of that result.

\begin{proposition}\label{pro:positroidfacets}
Let $M$ be a loopless matroid of rank $d$ with $m$ connected components, and consider some dimension $k$ with $n-d \leq k < n - m = \dim(\Gamma_M)$. If all the $k$-dimensional loopless faces of the matroid polytope $\Gamma_M$ are non-crossing matroids, then $M$ is a positroid.
\end{proposition}
\begin{proof}
We first prove the case where $k = n - m - 1 = \dim(\Gamma_M) - 1$. Assume by contradiction that all the loopless facets of the matroid polytope $\Gamma_M$ are non-crossing matroids, but $M$ is not a positroid. Lemma \ref{lem:noncrossing} tells us that $M$ is a non-crossing matroid. Let $M = M_1 \oplus M_2 \oplus \dots \oplus M_m$ be the decomposition of $M$ into connected components, and denote by $E_i$ the ground set of the matroid $M_i$. Since $M$ is not a positroid, \cite[Theorem 7.6]{ARW1} implies that some $M_i$ is not a positroid. By \cite[Proposition 5.6]{ARW1}, there is a facet $F$ of $\Gamma_{M_i}$ that is not a non-crossing matroid. This facet must correspond to a matroid with two connected components $M_i|S \oplus M_i/S$, where the sets $S$ and $E_i \setminus S$ cross. In particular, neither $M_i|S$ nor $M_i/S$ are loops, and so $F$ is a loopless matroid. But then the loopless facet of $\Gamma_M$ given by the matroid $M_1 \oplus \dots \oplus M_{i-1} \oplus M_i|S \oplus M_i/S \oplus M_{i+1} \oplus \dots \oplus M_m$ is not a non-crossing matroid, contradicting our assumption.

To prove the result for more general $k$, we note that Lemma \ref{lem:noncrossing} implies that, for $k \geq n-d$, if all the $k$-dimensional loopless faces of $\Gamma_M$ are non-crossing, then all the $(k+1)$-dimensional loopless faces of $\Gamma_M$ are also non-crossing. 
Using this fact, we can reduce the case of general $k$ to the case where $k = n - m - 1$, completing the proof.
\end{proof}

\begin{proof}[Proof of Theorem \ref{thm:mainequivalencepositroids}]
The implication \ref{it:hyperbolic}$\implies$\ref{it:maxcones} follows from Theorem \ref{thm:non-crossing}. To show that \ref{it:maxcones}$\implies$\ref{it:positroid}, recall that the maximal cones of $\mathcal B (M)$ are dual to the $(n-d)$-dimensional loopless faces of the matroid polytope $\Gamma_M$. Moreover, the affine span of the cone dual to such a face $F$ is $\spann(e_{E_1},e_{E_2},\dots,e_{E_d})$, where $E_1, E_2, \dots, E_d$ are the ground sets of the connected components of $F$. This implies that any $(n-d)$-dimensional loopless face of $\Gamma_M$ is a non-crossing matroid. Suppose that $M$ has $m$ connected components, so $\Gamma_M$ is an $(n-m)$-dimensional polytope. If $m < d$, we can apply Proposition \ref{pro:positroidfacets} to conclude that $M$ is a positroid. If $m = d$, then all the connected components of $M$ must be uniform matroids of rank one, and so in particular they are positroids. Since the ground sets of all these components form a non-crossing partition of $[n]$, it follows from \cite[Theorem 7.6]{ARW1} that $M$ is a positroid. 

Finally, to prove that \ref{it:positroid}$\implies$\ref{it:hyperbolic}, suppose $M$ is a positroid. By \cite[Proposition 3.5]{ARW1}, the dual matroid $M^*$ is also a positroid. If $L \in \Gr_{\geq 0}(d,n)$ is a linear space representing $M^*$, by Proposition~\ref{prop:linearstable} its orthogonal complement $L^\perp$ is a positively hyperbolic linear space. Then $\mathcal B (M) = \trop(L^\perp)$ is the tropicalization of a hyperbolic variety.
\end{proof}

\section{Questions and future directions}
\label{sec:questions}

Positively hyperbolic varieties and their tropicalizations provide a structure that 
encompasses both Bergman fans of positroids,  the ``hive cones'' considered by Speyer's approach to Horn's problem \cite{SpeyerHives}, 
and more generally M-convex functions coming from stable polynomials \cite{Branden2}. 
As such, we believe they are deserving of further study. 
While we have established some of the basic properties of these varieties and their tropicalizations, 
there are many interesting questions remaining, some of which we
briefly discuss here.

Let us first recall a question posted at the end of Section~\ref{sec:curves}:
\begin{question}
Is every (non-fan) tropical curve in $\R^n/(1,1,\dots,)$ satisfying
Theorem~\ref{thm:tropical} realizable set-theoretically as the
tropicalization of a positively hyperbolic variety?
\end{question} 

We can also hope to understand the lower dimensional cones of these tropical varieties. 
\begin{question}
Are lower dimensional cones of the tropicalization of a positively hyperbolic variety spanned by $0/1$ vectors with non-crossing supports?
\end{question}
Note that this is true for linear varieties, as discussed in Remark \ref{rem:maximal}.  
This suggests that the statement is likely to hold at least for \emph{tropically smooth} varieties. 

%
%
%

Another natural direction is to extend more of the theory of stable polynomials to this context.
As mentioned in the introduction, complex inhomogeneous stable polynomials, such as $x_1+x_2+i$, 
may not define positively hyperbolic hypersurfaces. One important step is possibly a modification of our main definition: 
\begin{question}
Is there a natural definition of positive hyperbolicity that includes all complex, affine hypersurfaces defined by a stable polynomial? 
\end{question}
Most likely this would consist of requiring the imaginary parts of points in $X$ to avoid only half of the orthants with $\varbar(z)<c$, 
choosing exactly one from a sign pattern and its negation. In particular, this property would no longer be preserved under negation. 

One central piece of the theory of stable polynomials is the characterization by Borcea and Br\"and\'en of 
linear operations on $\C[x_1, \hdots, x_n]$ that preserve stability.
A possible generalization of this is the the following. Let us call an ideal $I\subset \C[x_1, \hdots, x_n]$ positively hyperbolic if its variety is. 
\begin{question}
What linear operations on $\C[x_1, \hdots, x_n]$ preserve positive hyperbolicity of ideals? 
\end{question}
A polynomial $f+ig$ where $f,g\in \R[x_1, \hdots, x_n]$ is stable if and only if $f$ and $g$ are stable and $g$ interlaces $f$.
For univariate polynomials, this is known as the Hermite--Biehler Theorem. See, for example, \cite[\S2]{wagnerSurvey}. 
\begin{question}
Can one characterize positive hyperbolicity of complex varieties in
terms of the positive hyperbolicity of some real varieties associated
to it? 
\end{question}

There are also several interesting questions concerning tropical positively hyperbolic varieties purely as combinatorial objects.
For example, a distinction is made between ``tropicalized linear spaces'', which are obtained as tropicalizations of linear subspaces, 
and ``tropical linear spaces'', which are tropical objects satisfying the
necessary combinatorial conditions 
but not necessarily coming from a classical linear subspace. 

\begin{question}
What is the right combinatorial definition of a tropical positively hyperbolic variety? 
\end{question}

Recently, Br\"and\'en and Huh defined a generalization of homogeneous stable polynomials, called \emph{Lorentzian polynomials}, whose 
valuations achieve \emph{all} M-concave functions \cite{BH19}.  It would be very interesting to find an analogous generalization for varieties of higher codimension. 

\begin{question}
Is there a good notion of a {\em Lorentzian variety} in $K^n$? Do their tropicalizations realize all `tropical positively hyperbolic varieties'? 
\end{question}

\begin{figure}[h]
  \begin{tabular}{cr}
 \begin{minipage}{0.27\linewidth}
\begin{center}
        \includegraphics[scale=0.12]{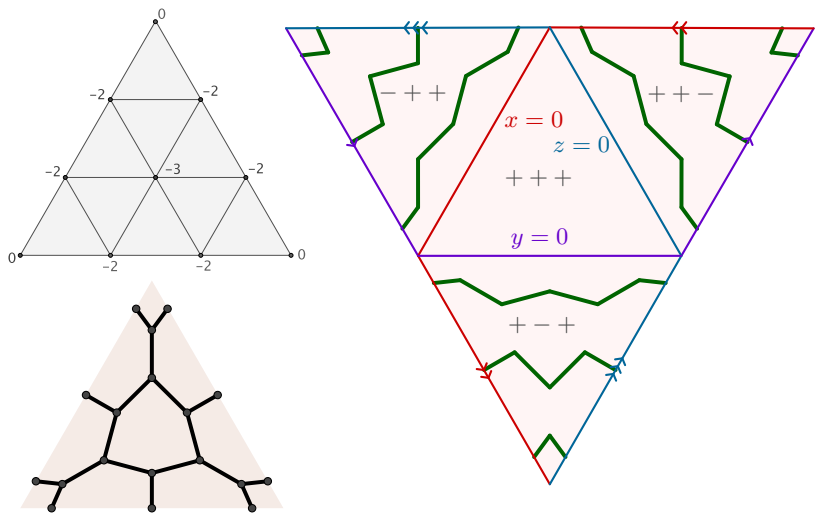}
\vspace{-2cm}
        \definecolor{uuuuuu}{rgb}{0.26666666666666666,0.26666666666666666,0.26666666666666666}
\definecolor{zzttqq}{rgb}{0.6,0.2,0.}
\begin{tikzpicture}[scale=0.7, line cap=round,line join=round,>=triangle 45,x=1.0cm,y=1.0cm]
\clip(-3.300924979593195,-2.5203020780753635) rectangle (10.964545923601321,6.17356037802241);
\fill[line width=1.5pt,color=zzttqq,fill=zzttqq,fill opacity=0.10000000149011612] (0.,0.) -- (5.,0.) -- (2.5,4.330127018922194) -- cycle;
\draw [line width=1.5pt] (2.4995823635152687,3.386321903125678)-- (2.4996942672199283,2.47897309625024);
\draw [line width=1.5pt] (2.4996942672199283,2.47897309625024)-- (3.167166871427041,1.8171953757817405);
\draw [line width=1.5pt] (3.167166871427041,1.8171953757817405)-- (3.410325845067579,0.9119632373275783);
\draw [line width=1.5pt] (3.410325845067579,0.9119632373275783)-- (2.499917552530058,0.6685088193419159);
\draw [line width=1.5pt] (2.499917552530058,0.6685088193419159)-- (1.5894492093238604,0.9119632373275783);
\draw [line width=1.5pt] (1.5894492093238604,0.9119632373275783)-- (1.8323848976542678,1.8171953757817405);
\draw [line width=1.5pt] (1.8323848976542678,1.8171953757817405)-- (2.4996942672199283,2.47897309625024);
\draw [line width=1.5pt] (2.4995823635152687,3.386321903125678)-- (2.79753988508124,3.789981368885757);
\draw [line width=1.5pt] (3.167166871427041,1.8171953757817405)-- (3.7497346608056836,2.1514497854526984);
\draw [line width=1.5pt] (3.410325845067579,0.9119632373275783)-- (4.201141137154043,0.45828883388985897);
\draw [line width=1.5pt] (4.201141137154043,0.45828883388985897)-- (4.7019294365301265,0.5129182020196391);
\draw [line width=1.5pt] (4.201141137154043,0.45828883388985897)-- (4.403985389889412,0.);
\draw [line width=1.5pt] (2.499917552530058,0.6685088193419159)-- (2.5,0.);
\draw [line width=1.5pt] (1.5894492093238604,0.9119632373275783)-- (0.7987458209420546,0.45828883388985897);
\draw [line width=1.5pt] (0.7987458209420546,0.45828883388985897)-- (0.5960146101105878,0.);
\draw [line width=1.5pt] (0.7987458209420546,0.45828883388985897)-- (0.2979440466407149,0.5129182020196391);
\draw [line width=1.5pt] (1.8323848976542678,1.8171953757817405)-- (1.2497346608056836,2.1514497854526984);
\draw [line width=1.5pt] (2.201525274970652,3.789981368885757)-- (2.4995823635152687,3.386321903125678);
\begin{scriptsize}
\draw [fill=uuuuuu] (2.4995823635152687,3.386321903125678) circle (2.0pt);
\draw [fill=uuuuuu] (4.201141137154043,0.45828883388985897) circle (2.0pt);
\draw [fill=uuuuuu] (0.7987458209420546,0.45828883388985897) circle (2.0pt);
\draw [fill=uuuuuu] (2.4996942672199283,2.47897309625024) circle (2.0pt);
\draw [fill=uuuuuu] (3.410325845067579,0.9119632373275783) circle (2.0pt);
\draw [fill=uuuuuu] (1.5894492093238604,0.9119632373275783) circle (2.0pt);
\draw [fill=uuuuuu] (3.167166871427041,1.8171953757817405) circle (2.0pt);
\draw [fill=uuuuuu] (2.499917552530058,0.6685088193419159) circle (2.0pt);
\draw [fill=uuuuuu] (1.8323848976542678,1.8171953757817405) circle (2.0pt);
\draw [fill=uuuuuu] (3.7497346608056836,2.1514497854526984) circle (2.0pt);
\draw [fill=uuuuuu] (1.2497346608056836,2.1514497854526984) circle (2.0pt);
\draw [fill=uuuuuu] (2.5,0.) circle (2.0pt);
\draw [fill=uuuuuu] (2.79753988508124,3.789981368885757) circle (2.0pt);
\draw [fill=uuuuuu] (4.7019294365301265,0.5129182020196391) circle (2.0pt);
\draw [fill=uuuuuu] (0.5960146101105878,0.) circle (2.0pt);
\draw [fill=uuuuuu] (4.403985389889412,0.) circle (2.0pt);
\draw [fill=uuuuuu] (2.201525274970652,3.789981368885757) circle (2.0pt);
\draw [fill=uuuuuu] (0.2979440466407149,0.5129182020196391) circle (2.0pt);
\end{scriptsize}
\end{tikzpicture}
\end{center}
    \end{minipage}
&
    \begin{minipage}{0.55\linewidth}
\vspace{-2cm}
      \begin{center}    
\definecolor{qqwwzz}{rgb}{0.,0.4,0.6}
\definecolor{ccqqqq}{rgb}{0.8,0.,0.}
\definecolor{wwqqcc}{rgb}{0.4,0.,0.8}
\definecolor{qqwuqq}{rgb}{0.,0.39215686274509803,0.}
\definecolor{ffzzzz}{rgb}{1,0.6,0.6}
\definecolor{uuuuuu}{rgb}{0.3,0.3,0.3}
\begin{tikzpicture}[scale=0.7, line cap=round,line join=round,>=triangle 45,x=1.0cm,y=1.0cm]
\clip(-3.9869600332105346,-4.894805284241832) rectangle (13.250198005717849,5.2035674913944945);
\fill[line width=0.8pt,color=ffzzzz,fill=ffzzzz,fill opacity=0.10000000149011612] (0.,0.) -- (5.,0.) -- (2.5,4.330127018922194) -- cycle;
\fill[line width=0.8pt,color=ffzzzz,fill=ffzzzz,fill opacity=0.10000000149011612] (7.512048195851799,4.32314860509302) -- (5.,0.) -- (2.5120675815800837,4.337071855685034) -- cycle;
\fill[line width=0.8pt,color=ffzzzz,fill=ffzzzz,fill opacity=0.10000000149011612] (0.,0.) -- (-2.507959845445562,4.32552163485893) -- (2.492031697684249,4.334717755256618) -- cycle;
\fill[line width=0.8pt,color=ffzzzz,fill=ffzzzz,fill opacity=0.10000000149011612] (0.,0.) -- (5.,0.) -- (2.5,-4.330127018922193) -- cycle;
\draw [line width=1.5pt,color=qqwuqq] (1.701425245200096,3.8775630875932663)-- (1.9042401060301457,4.339779937959488);
\draw [line width=1.5pt,color=qqwuqq] (1.701425245200096,3.8775630875932663)-- (0.9095385227732811,3.4180899960291566);
\draw [line width=1.5pt,color=qqwuqq] (0.6667313978433427,2.505609013742603)-- (0.9095385227732811,3.4180899960291566);
\draw [line width=1.5pt,color=qqwuqq] (-0.001397417434056858,1.8369464859070654)-- (0.6667313978433427,2.505609013742603);
\draw [line width=1.5pt,color=qqwuqq] (-0.001397417434056858,1.8369464859070654)-- (-7.022441038171956E-4,0.9231206132951049);
\draw [line width=1.5pt,color=qqwuqq] (-7.022441038171956E-4,0.9231206132951049)-- (-0.2989561418912493,0.5156148181450715);
\draw [line width=1.5pt,color=qqwuqq] (-0.0027845206859944405,3.6603346748173697)-- (-0.003296704642941295,4.333615612153503);
\draw [line width=1.5pt,color=qqwuqq] (-0.0027845206859944405,3.6603346748173697)-- (-0.914734514625107,3.4121947423707852);
\draw [line width=1.5pt,color=qqwuqq] (-0.914734514625107,3.4121947423707852)-- (-0.6705402864432308,2.5012875357470343);
\draw [line width=1.5pt,color=qqwuqq] (-0.6705402864432308,2.5012875357470343)-- (-1.253979922722781,2.162760817429465);
\draw [line width=1.5pt,color=qqwuqq] (-1.7073164103821616,3.866547523418634)-- (-1.9108335153160279,4.327451286347516);
\draw [line width=1.5pt,color=qqwuqq] (-1.7073164103821616,3.866547523418634)-- (-2.2090037035543126,3.8099068167138586);
\draw [line width=1.5pt,color=qqwuqq] (3.0986543432379055,4.339497065877348)-- (3.3011609250207803,3.877310343362015);
\draw [line width=1.5pt,color=qqwuqq] (3.3011609250207803,3.877310343362015)-- (4.092902422466992,3.4175870524042975);
\draw [line width=1.5pt,color=qqwuqq] (4.092902422466992,3.4175870524042975)-- (4.335057921455794,2.5052403341346015);
\draw [line width=1.5pt,color=qqwuqq] (4.335057921455794,2.5052403341346015)-- (5.002843118003373,1.8364435422822067);
\draw [line width=1.5pt,color=qqwuqq] (5.002843118003373,1.8364435422822067)-- (5.001428751928856,0.9228678690638532);
\draw [line width=1.5pt,color=qqwuqq] (5.001428751928856,0.9228678690638532)-- (5.299443485205924,0.5153319460629298);
\draw [line width=1.5pt,color=qqwuqq] (5.005665251270074,3.659332499738415)-- (5.006707316005739,4.332429097270548);
\draw [line width=1.5pt,color=qqwuqq] (5.9176690915411125,3.410827599854592)-- (5.005665251270074,3.659332499738415);
\draw [line width=1.5pt,color=qqwuqq] (5.67269145928561,2.5002853606680797)-- (5.9176690915411125,3.410827599854592);
\draw [line width=1.5pt,color=qqwuqq] (5.67269145928561,2.5002853606680797)-- (6.256024097925899,2.16157430254651);
\draw [line width=1.5pt,color=qqwuqq] (6.710824955061843,3.8646799821152285)-- (7.212604710645876,3.80781665903009);
\draw [line width=1.5pt,color=qqwuqq] (6.710824955061843,3.8646799821152285)-- (6.914760288773572,4.325361128663749);
\draw [line width=1.5pt,color=qqwuqq] (0.798947876881259,-0.4624223717166378)-- (0.29817018823828434,-0.5175444696336591);
\draw [line width=1.5pt,color=qqwuqq] (1.5898512867362076,-0.9201886931086616)-- (0.798947876881259,-0.4624223717166378);
\draw [line width=1.5pt,color=qqwuqq] (2.5002122928059927,-0.6745384371024677)-- (1.5898512867362076,-0.9201886931086616);
\draw [line width=1.5pt,color=qqwuqq] (3.4107279224799263,-0.9201886931086611)-- (2.5002122928059927,-0.6745384371024677);
\draw [line width=1.5pt,color=qqwuqq] (3.4107279224799263,-0.9201886931086611)-- (4.201343193093247,-0.46242237171663697);
\draw [line width=1.5pt,color=qqwuqq] (4.201343193093247,-0.46242237171663697)-- (4.702155578127695,-0.5175444696336581);
\draw [line width=1.5pt,color=qqwuqq] (1.8331860847904569,-1.8335855761728037)-- (1.2506832180798395,-2.170854794724038);
\draw [line width=1.5pt,color=qqwuqq] (1.8331860847904569,-1.8335855761728037)-- (2.500787226943546,-2.5013322032307523);
\draw [line width=1.5pt,color=qqwuqq] (2.500787226943546,-2.5013322032307523)-- (3.1679680585632295,-1.8335855761728035);
\draw [line width=1.5pt,color=qqwuqq] (3.1679680585632295,-1.8335855761728035)-- (3.7506832180798395,-2.1708547947240375);
\draw [line width=1.5pt,color=qqwuqq] (2.2031962479213947,-3.824165119814417)-- (2.501075366185172,-3.4168648460147995);
\draw [line width=1.5pt,color=qqwuqq] (2.501075366185172,-3.4168648460147995)-- (2.7992108580319823,-3.8241651198144164);
\draw [line width=0.8pt,color=wwqqcc] (5.,0.)-- (7.512048195851799,4.32314860509302);
\draw [line width=0.8pt,color=wwqqcc] (6.298916173719263,2.2353900907713675) -- (6.350930111357858,2.10642734795504);
\draw [line width=0.8pt,color=wwqqcc] (6.298916173719263,2.2353900907713675) -- (6.161118084493937,2.216721257137979);
\draw [line width=0.8pt,color=wwqqcc] (-2.507959845445562,4.32552163485893)-- (0.,0.);
\draw [line width=0.8pt,color=wwqqcc] (-1.211157653645618,2.088904510805021) -- (-1.3489380312399233,2.107703614330255);
\draw [line width=0.8pt,color=wwqqcc] (-1.211157653645618,2.088904510805021) -- (-1.159021814205639,2.2178180205286746);
\draw [line width=0.8pt,color=ccqqqq] (7.512048195851799,4.32314860509302)-- (2.501366436159678,4.341709589448077);
\draw [line width=0.8pt,color=ccqqqq] (4.835963054049218,4.333061582374747) -- (4.921741782594464,4.442509508223267);
\draw [line width=0.8pt,color=ccqqqq] (4.835963054049218,4.333061582374747) -- (4.920928587460493,4.222981171422029);
\draw [line width=0.8pt,color=ccqqqq] (5.0067073160057385,4.332429097270548) -- (5.092486044550984,4.441877023119067);
\draw [line width=0.8pt,color=ccqqqq] (5.0067073160057385,4.332429097270548) -- (5.091672849417012,4.222348686317829);
\draw [line width=0.8pt,color=ccqqqq] (0.,0.)-- (2.501366436159679,-4.341709589448076);
\draw [line width=0.8pt,color=ccqqqq] (1.335919915700631,-2.318803124918012) -- (1.1981919260512512,-2.299623979720105);
\draw [line width=0.8pt,color=ccqqqq] (1.335919915700631,-2.318803124918012) -- (1.388411207729219,-2.1900339399219453);
\draw [line width=0.8pt,color=ccqqqq] (1.2506832180798395,-2.1708547947240384) -- (1.1129552284304598,-2.1516756495261316);
\draw [line width=0.8pt,color=ccqqqq] (1.2506832180798395,-2.1708547947240384) -- (1.3031745101084276,-2.0420856097279714);
\draw [line width=0.8pt,color=qqwwzz] (2.501366436159679,-4.341709589448076)-- (5.,0.);
\draw [line width=0.8pt,color=qqwwzz] (3.793266593484114,-2.096860491589949) -- (3.84581875068081,-2.2256048488152484);
\draw [line width=0.8pt,color=qqwwzz] (3.793266593484114,-2.096860491589949) -- (3.6555476854788673,-2.116104740632827);
\draw [line width=0.8pt,color=qqwwzz] (3.708099842675563,-2.2448490978581273) -- (3.760651999872259,-2.3735934550834266);
\draw [line width=0.8pt,color=qqwwzz] (3.708099842675563,-2.2448490978581273) -- (3.5703809346703164,-2.264093346901005);
\draw [line width=0.8pt,color=qqwwzz] (3.878433344292664,-1.948871885321772) -- (3.93098550148936,-2.0776162425470712);
\draw [line width=0.8pt,color=qqwwzz] (3.878433344292664,-1.948871885321772) -- (3.7407144362874174,-1.9681161343646498);
\draw [line width=0.8pt,color=qqwwzz] (2.501366436159678,4.341709589448077)-- (-2.507959845445562,4.32552163485893);
\draw [line width=0.8pt,color=qqwwzz] (-0.08866897557499566,4.333339726262402) -- (-0.0036514150743587885,4.443379960494714);
\draw [line width=0.8pt,color=qqwwzz] (-0.08866897557499566,4.333339726262402) -- (-0.002941994211525561,4.223851263812292);
\draw [line width=0.8pt,color=qqwwzz] (0.08207556628911176,4.333891498044605) -- (0.16709312678974864,4.443931732276918);
\draw [line width=0.8pt,color=qqwwzz] (0.08207556628911176,4.333891498044605) -- (0.16780254765258185,4.224403035594495);
\draw [line width=0.8pt,color=qqwwzz] (-0.2594135174391026,4.332787954480198) -- (-0.17439595693846574,4.442828188712511);
\draw [line width=0.8pt,color=qqwwzz] (-0.2594135174391026,4.332787954480198) -- (-0.17368653607563253,4.223299492030088);
\draw [line width=0.8pt,color=ccqqqq] (2.501366436159678,4.341709589448077)-- (0.,0.);
\draw [line width=0.8pt,color=qqwwzz] (5.,0.)-- (2.501366436159678,4.341709589448077);
\draw [line width=0.8pt,color=wwqqcc] (5.,0.)-- (0.,0.);
\begin{scriptsize}
\draw[color=uuuuuu] (2.2899673281350843,1.471560161268026) node {$+++$};
\draw[color=uuuuuu] (5,3.0814456762245417) node {$++-$};
\draw[color=uuuuuu] (-0.1,3.097707146072587) node {$-++$};
\draw[color=uuuuuu] (2.355013207527267,-1.309151182747774) node {$+-+$};
\draw[color=ccqqqq] (2.2,2.6) node {$x=0$};
\draw[color=qqwwzz] (3.1030408205373665,2.1220189551898505) node {$z=0$};
\draw[color=wwqqcc] (2.30622879798313,0.3332572719048329) node {$y=0$};
\end{scriptsize}
\end{tikzpicture}
        \end{center}
    \end{minipage}
\end{tabular}
\vspace{-2cm}
\caption{The top left figure is a subdivision of the Newton polytope
  corresponding to a stable polynomial, homogeneous of degree 3 in
  three variables, with positive coefficients. The labels show the
  valuations of the coefficients.  They form an $M$-convex function on
  these lattice points because all the edges of the corresponding regular subdivision are in
  directions $e_i - e_j$.  The bottom left figure is the
  exponentiation of the tropical curve defined by the stable
  polynomial, and the figure on the right is the {\em real
    tropicalization} in $\R\P^2$ (as in~\cite{JSY}) of the positively hyperbolic curve. Topologically it
  consists of nested ovals.
Compare with figures
  in~\cite{SpeyerHives}. }
  \label{fig:honeycomb}
\end{figure}

It is also natural to consider signed analogues of tropical varieties, as in Figure~\ref{fig:honeycomb}. 
We can call a signed tropical variety of codimension $c$ positively hyperbolic if 
for every positive tropical linear space $L$ of dimension $c-1$, 
the tropical variety meets every signed tropical linear space of dimension $c$ containing $L$
in the correct number of real points. For example, this is satisfied by the tropicalization of the cubic curve in  Example~\ref{ex:reciprocalPlane}, 
as seen in Figure~\ref{fig:SignedRecipLine}.
\begin{figure}[ht]
\begin{center}
\includegraphics[height=2in]{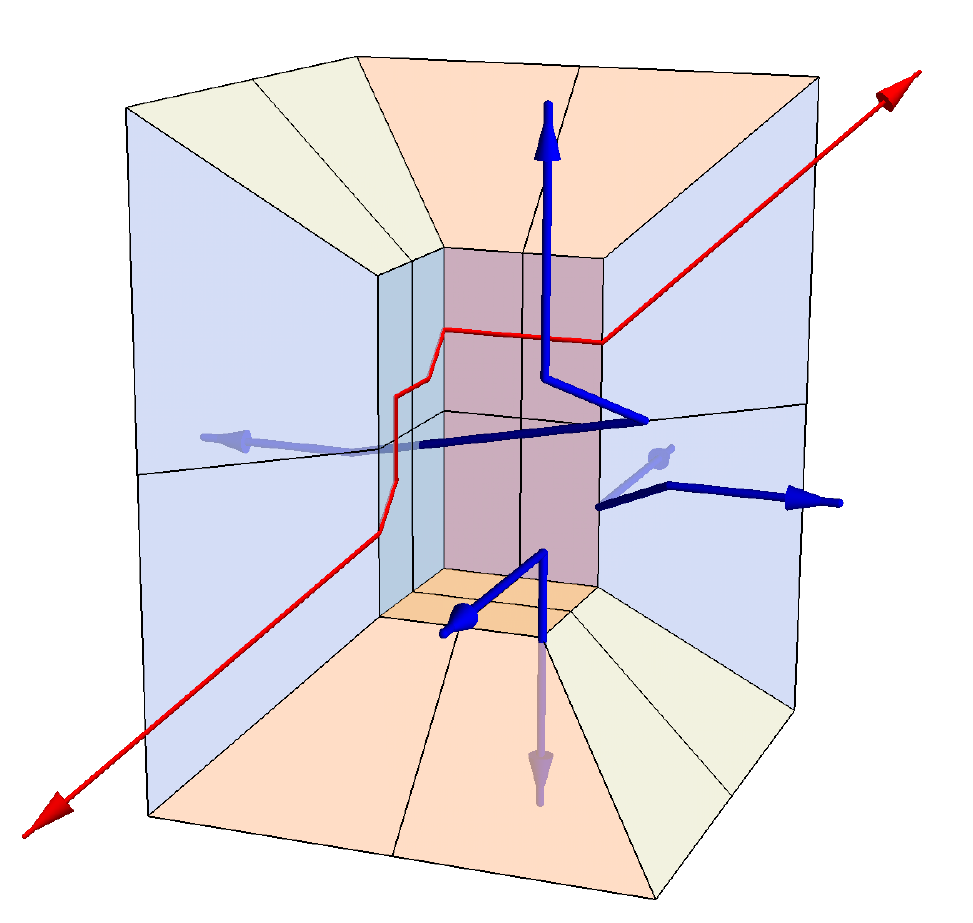}
\end{center}
\caption{The signed tropical variety of the reciprocal plane in Example~\ref{ex:reciprocalPlane} (blue) 
showing three real intersection points with a real tropical plane containing the tropicalization of 
a line (red), with respect to which it is hyperbolic.} 
\label{fig:SignedRecipLine}
\end{figure}
 
\begin{question}
What combinatorial or topological properties of signed tropical varieties are implied by tropical positive hyperbolicity? 
\end{question}

Another direction of interest is the study of which of these tropical objects can be realized as tropicalizations of positively hyperbolic varieties. 
One concrete instance of this would be the following generalization of the titular result of \cite{ARW2}:

\begin{question}
Are all positively oriented valuated matroids realizable as the tropicalization of a positive linear subspace in $K^n$?
\end{question}

%
%

\bibliography{PosHypRefs}
\bibliographystyle{alpha}

\end{document}